\documentclass{amsart}
\usepackage{graphicx} 
\usepackage{amsmath,amsthm,amsfonts,hyperref} 
\usepackage{amssymb}
\usepackage{enumerate, parskip}
\theoremstyle{plain}
\newtheorem{theorem}{Theorem}[section]
\newtheorem{lemma}[theorem]{Lemma}

\newtheorem{proposition}[theorem]{Proposition}
\theoremstyle{definition}
\newtheorem{definition}[theorem]{Definition}

\newtheorem{remark}{Remark}
\newcommand{\RE}{\text{Ritt}_{\text{E}}}

\title{$H^\infty$ Functional Calculus for a Commuting pair of \( \RE \) Operators}

\author[Suman Mondal]{Suman Mondal}
\address[Suman Mondal]{School of Mathematics, Indian Institute of Science Education and Research Thiruvananthapuram, Kerala - 695551}
\email{suman2024@iisertvm.ac.in}

\author[Subhajit Palai]{Subhajit Palai}
\address[Subhajit Palai]{School of Mathematics, Indian Institute of Science Education and Research Thiruvananthapuram, Kerala - 695551}
\email{subhajit22@iisertvm.ac.in}

\author[Samya Kumar Ray]{Samya Kumar Ray}
\address[Samya Kumar Ray]{The Institute of Mathematical Sciences, 4th Cross Street, CIT Campus, Tharamani
Chennai, Tamil Nadu 600113, India and Homi Bhabha National Institute,
Mumbai, India }
\email{samya@imsc.res.in}

\begin{document}
\begin{abstract}
In this article, we develop a framework for the joint functional calculus of commuting pair of \( \RE \) operators on Banach spaces. We establish a transfer principle that relates the bounded holomorphic functional calculus for pair of $\RE$ operators to that of their associated sectorial counterparts. In addition, we prove a joint dilation theorem for commuting pair of \( \RE \) operators on a broad class of Banach spaces. As a key application, we obtain an equivalent set of criteria on \( L^p \)-spaces for \( 1 < p < \infty \) that determine when a commuting pair of $\RE$ operators admits a joint bounded functional calculus.
\end{abstract}
\maketitle

\section{Introduction}
Let \( X \) be a Banach space and \( T \in \mathcal{B}(X) \) a bounded linear operator. In many important problems involving operator theory, harmonic analysis, and partial differential equations, it is crucial to assign a rigorous meaning to expressions of the form \( f(T) \), where \( f \) is a bounded holomorphic function defined on a suitable domain in the complex plane. The foundational framework for this is provided by the theory of \emph{bounded holomorphic functional calculus}, which generalizes the polynomial calculus in a way that preserves boundedness and algebraic structure. Moreover, in the setting of Hilbert spaces, one of the cornerstone results in this direction is the \emph{von Neumann inequality} \cite{MR43386}, which says that for any contraction \( T \) on a Hilbert space and \( f \) a bounded holomorphic function on an open set containing the closed unit disc \( \overline{\mathbb{D}} \), we have
\[
\|f(T)\| \leq \|f\|_{\infty, \mathbb{D}}.
\]
This inequality ensures that the operator \( f(T) \) is well-defined and bounded, and it leads naturally to the development of an \( H^\infty \)-functional calculus for contractions. The scope of this inequality was extended to tuple of commuting contractions by Ando \cite{MR155193}, who showed that a commuting tuple of contractions \( (T_1, T_2) \) on a Hilbert space also satisfies a joint version of the von Neumann inequality, that is, for any \( f \) a bounded holomorphic function on an open set containing the closed unit disc \( \overline{\mathbb{D}^2} \), we have
\[
\|f(T_1,T_2)\| \leq \|f\|_{\infty, \mathbb{D}^2}.
\] This gives rise to a bounded $H^\infty$-functional calculus for a commuting pair of contractions on Hilbert spaces. We refer the reader to \cite{MR275190} for various applications of the above-mentioned results.

However, the multivariate extension of this result meets a fundamental obstruction. In a striking result, Varopoulos together with Kaijser \cite{MR355642}  (also see \cite{MR3897976}) constructed an explicit example of a triple of commuting contractions on a Hilbert space for which the von Neumann inequality fails. Interestingly, it remains an open question whether the von Neumann inequality holds up to a universal constant for triples of commuting contractions. For a comprehensive discussion of this long-standing problem, we refer the reader to \cite{MR1818047}[Chapter 1] and \cite{MR1976867}[Chapter 5]. This demonstrates that in the multivariate setting, the naive extension of bounded functional calculus breaks down even in the Hilbert space setting. It also illustrates the delicate nature of multivariate operator theory, where phenomena in higher dimensions can diverge sharply from the two-variable case.

In contrast to contractions on Hilbert spaces, the theory of \emph{sectorial operators} on Banach spaces offers a robust and flexible framework for defining a functional calculus. The development of a bounded holomorphic functional calculus for sectorial operators, particularly the \( H^\infty \)-functional calculus was pioneered by McIntosh and collaborators \cite{MR912940, MR1364554}. This theory has deep applications in the analysis of partial differential equations, harmonic analysis, and the theory of semigroups \cite{MR2244037}. A particularly important extension is the notion of \emph{joint functional calculus} for families of sectorial operators with \emph{commuting resolvents}. This concept, introduced in \cite{albrecht1994functional}, provides a framework for joint \( H^\infty \)-functional calculus and plays a pivotal role in studying systems of evolution equations, especially those involving \emph{maximal regularity} on Banach spaces \cite{MR1635157}.

Motivated by the success of sectorial operator theory, one naturally seeks analogous results in the discrete setting. This leads to the study of \emph{Ritt operators}, which are discrete analogues of sectorial operators and arise naturally in the context of discrete semigroups, ergodic theory, and time-discretized evolution equations (see \cite{MR2980915,MR3060752}). Christian Le Merdy initiated a systematic study of the functional calculus for Ritt operators in \cite{MR3293430}. Together with his collaborators, further results were developed in \cite{MR3265289,MR3683097,MR3293430}, drawing analogies with the theory of sectorial operators and providing characterizations of Ritt operators on $L^p$-spaces that admit a bounded $H^\infty$-functional calculus. The joint functional calculus for commuting tuple of Ritt operators was introduced in \cite{MR3928691} and \cite{MR4043874}. Moreover, many multivariate anlogoues of the above-mentioned results were also obtained.

Recently, in \cite{MR4819960}, the class of Ritt operators was generalized to the so-called \( \RE \) operators. These operators allow for spectra that intersect the unit circle in finitely many points, provided that suitable decay conditions are met near those spectral points. Le Merdy and Bouabdillah established several fundamental properties of the \( H^\infty \)-functional calculus for \( \RE \) operators, including appropriate square function estimates \cite{bouabdillah2024squarefunctionsassociatedritte}.
 In this paper, we develop a joint functional calculus for commuting pair of $\RE$ operators. Our first result is the following multivariate transfer principle.
\begin{theorem}\label{transferprin}(Transfer Principle): Let $E=\{\xi_1,\dots,\xi_N\}$ be a finite subset of $\mathbb{T}$. Let $(T_1,T_2)$ be a commuting pair of $\RE$ operators on a Banach space $X$. For any $i=1,...,N $, denote $A^1_i=I_{X}-\overline{\xi}_i T_1$ and $A^2_i=I_{X}-\overline{\xi}_i T_2$. Then the following statements are equivalent. \begin{enumerate}
    \item There exist $\theta_i,\theta_j\in(0,\frac{\pi}{2})$, $1\leq i,j\leq N$ such that $(A^1_i,A^2_j)$ admits bounded $H^{\infty}(\Sigma_{\theta_{i}}\times\Sigma_{\theta_{j}})$-functional calculus \item There exists $s_1,s_2\in(0,1)$ such that $(T_1,T_2)$ admits bounded $H^{\infty}(E_{s_1}\times E_{s_2})$-functional calculus.
\end{enumerate}
\end{theorem}
We refer the reader to Section \ref{prelim} for any unexplained notations. To prove Theorem \ref{transferprin}, we take a slightly different route from \cite{MR4819960} in the multivariate setting and our proof follows \cite{MR3928691} more closely but requires a few more technical lemmas. Our next theorem is the following multivariate dilation theorem. This generalizes some results in \cite{MR3928691}.
\begin{theorem}\label{dilationthm}Let $ 1< p<\infty$ and
     $X$ be a reflexive Banach space such that both $X$ and $X^*$ have finite cotype. Let $E=\{\xi_1,\dots,\xi_N\}$ be a finite subset of $\mathbb{T}$. Suppose $(T_1,T_2)$ is a commuting pair of bounded operators on $X$ such that each $T_i$ is a $\RE$ operator of type $r_i$ with bounded $H^{\infty}(E_{s_i})$ functional calculus. Then there exist a measure space ${\Omega} '$, a commuting pair of isometries $(U_1,U_2)$ on $L^p(\Omega',X)$, together with two bounded maps $J:X\rightarrow L^p(\Omega',X)$ and $Q:L^p({\Omega'},X) \rightarrow X$,  such that 
    $$T^{i_1}_1T^{i_2}_2=QU^{i_1}_1U^{i_2}_2J$$ 
    for all $i_1,i_2 \geq0$.
\end{theorem}
The existence of a bounded functional calculus is intimately tied to dilation theory. Sz.-Nagy and Foias \cite{MR3897976} showed that any contraction on a Hilbert space admits a unitary dilation, and Ando \cite{MR155193} extended this to pairs of commuting contractions. These results yield von Neumann inequalities in one and two variables. However, the failure of the von Neumann inequality in three variables shows that three commuting contractions may lack a joint unitary dilation. Beyond Hilbert spaces, Akcoglu and Sucheston \cite{MR458230} proved that positive contractions on $L^p$-spaces ($1<p\neq 2<\infty$) have onto isometric dilations, which implies that Matsaev’s conjecture holds for this class. For further developments, see \cite{MR481928}, \cite{MR615568}, \cite{MR4611833} and \cite{MR4092689} for multivariate cases. Recent work in \cite{MR3265289} and\cite{MR3683097} characterizes bounded functional calculus for Ritt operators via the notion of loose dilation, which was further generalized in \cite{MR3928691} in the multivariate setting. In this article, we prove the following theorem in the context of $\text{Ritt}_E$ operators.
\begin{theorem}\label{CLASS} Let $E=\{\xi_1,\dots,\xi_N\}$ be a finite subset of $\mathbb{T}$. Let $1<p\neq 2<\infty$ and $(T_1,T_2)$ be a commuting pair of $\text{Ritt}_E$ operators on $L^p(\Omega).$ Then the following assertions are equivalent.
\begin{itemize}
\item[1.]There exist $s_i\in(0,\frac{\pi}{2}),$ $1\leq i\leq 2,$ such that $(T_1,T_2)$ admits a joint bounded $H^\infty(
E_{s_1}\times E_{s_2})$-functional calculus
\item[2.] Each $T_i,$ $1\leq i\leq 2$ is $R$-$\RE$ and $(T_1,T_2)$ admits a joint isometric loose dilation.
\item[3.] Each $T_i,$ $1\leq i\leq 2$ is $R$-$\RE$ and $(T_1,T_2)$ is jointly $p$-completely polynomially bounded.
\item[4.] Each $T_i,$ $1\leq i\leq 2$ is $R$-$\RE$ and $(T_1,T_2)$ is jointly $p$-polynomially bounded.
\item[5.] Each $T_i,$ $1\leq i\leq 2$ is $R$-$\RE$ and $I-\overline{\xi_j}T_i$ admits a bounded $H^\infty(\Sigma_{\theta_{ij}})$-functional calculus for $\theta_{ij}\in(0,\pi)$,  for each $1\leq i\leq 2$ and $1\leq j\leq N.$
\end{itemize}
\end{theorem}
We prove an appropriate analogue of the above theorem in the context of Hilbert space, which in turn provides an alternate proof of some of the results in \cite{MR4861029}. In the main point where we differ is the use of transfer principle (Theorem \ref{transferprin}) instead of Franks-McIntosh type decomposition. 

For notational simplicity, we present all the results and proofs for two commuting operators. However, all of our results are valid for arbitrary tuple of commuting operators.

The plan of the paper is the following. In Section \ref{prelim}, we introduce all the necessary definitions and a few technical lemmas which will be useful in later sections. In Section \ref{jointritte} we introduce the notion of joint functional calculus for commuting tuple of $\text{Ritt}_E$ operators and prove the transfer principle, i.e. Theorem \ref{transferprin} in Section \ref{sectran}. In the last Section \ref{alldithm}, we prove Theorem \ref{dilationthm} and Theorem \ref{CLASS}. 
 \section{Notation and Definition}\label{prelim}
 Let $X$ be a Banach space. We denote $B(X)$ as the set of all bounded linear transformations from $X$ to $X$. Let us denote $\sigma(T)$ the spectrum of $T.$ For $a,b>0$,
 we write $a\lesssim b$ if $a\leq cb$, where $c>0$ independent of $a$ and $b$. Let $(\Omega, \mathcal{F}, \mu)$ be a $\sigma$-finite measure space and let $1 \leq p \leq \infty$. We denote $L^p(\Omega, X)$ to be the usual Bochner space. We denote by $X\oplus^p X\oplus^pX\dots\oplus^pX=\{(x_1,\dots,x_N): x_i\in X,1\leq i \leq N\}$ equipped with the norm $\lVert (x_1,\dots,x_N)\rVert=\left(\sum\limits^N_{i=1}\lVert x_i\rVert^p\right)^{\frac{1}{p}}$. Note that  we can identify $X\oplus^p X\oplus^pX\dots\oplus^pX$ with $L^p([N],X)$, where $[N]=\{1,2,\dots,N\}$ is measure space with discrete measure. Let $S\in B(X)$, define $\oplus^NS=\underbrace{S\oplus S\dots\oplus S}_{N\ \text{times}}$. For any set $V$ and $f:V\rightarrow\mathbb{C}$, we define $\lVert f\rVert_{\infty,V}=sup\{\lvert f(z)\rvert: z\in V\}.$ Let us define $\mathbb{C}[z]$ as the set of one-variable polynomials with complex coefficients. Similarly, we define $\mathbb{C}[z_1,z_2]$ as the set of two-variable polynomials with complex coefficients.  
 \begin{definition}[Sector] Let  $\omega \in(0,\pi)$, the set $\Sigma_\omega=\{z\in\mathbb{C}:\lvert Arg(z) \rvert<\omega\}$ is called open sector of angle $2\omega$ and the set $\overline{\Sigma}_{\omega}=\{z\in\mathbb{C}:\lvert Arg(z) \rvert \leq \omega\}$ is called closed sector of angle $2\omega$.  
 \end{definition}
 \begin{definition}[Sectorial Operator]Let any $\omega \in(0,\pi).$ An operator $T\in B(X)$ is said to be $\omega$-sectorial operator if
 \begin{enumerate}
     \item [(i)]$\sigma(T)\subseteq \overline\Sigma_\omega$ .
     \item[(ii)] For all $\theta\in(\omega,\pi)$, there exits a $C_\theta\geq0$ such that $\lVert zR(T,z)\rVert\leq C_{\theta}$ for all $z\in \mathbb{C}\setminus\overline{\Sigma}_\theta$.
 \end{enumerate}
Let $\theta\in(0,\frac{\pi}{2})$, consider the set $H^{\infty}_{0}(\Sigma_{\theta})$ which is the collection of  all bounded holomorphic functions such that
\begin{equation*}   
|f(z)| \lesssim \frac{|z|^{s}}{1 + |z|^{2s}}, \text{for some}\hspace{0.1 cm}s>0. 
\end{equation*}
We recommend to interested readers \cite{MR3293430} about bounded functional calculus for sectorial operators.
\end{definition}

 Let $\theta \in (0,\pi)$ and $\Gamma_{\theta}$ be the boundary of $\Sigma_{\theta}$, oriented counter-clockwise . For $\theta_i \in (0,\pi)$, $1 \leq i \leq 2$, we denote $H^{\infty}_0 \Big( \Sigma_{\theta_1}\times \Sigma_{\theta_2} \Big)$ to be the set of all bounded holomorphic functions $f:  \Sigma_{\theta_1}\times\Sigma_{\theta_2} \to \mathbb{C}$ with the property that there exist constants $C,s_1,s_2 > 0$, depending only on $f$, such that
 \[
|f(z_1,z_2)| \leq C  \frac{|z_1|^{s_1}}{1 + |z_1|^{2s_1}}.\frac{|z_2|^{s_2}}{1 + |z_2|^{2s_2}} 
\quad \text{for all } (z_1,z_2) \in \Sigma_{\theta_1}\times\Sigma_{\theta_2}.
\]
Let $(A_1, A_2)$ be a pair of commuting operators  such that each $A_i$ is sectorial of type $\omega_i\in (0, \pi)$, $i=1,2$. Let $\omega_1 < \theta_1< \pi$ and $ \omega_2 <\theta_2<\pi$. We define
\begin{equation}
    f(A_1,A_2) := \left( \frac{1}{2\pi i} \right)^2 \int_{ \Gamma_{\theta_1}\times\Gamma_{\theta_2}} f(z_1,z_2) R(z_1, A_1)R(z_2,A_2) dz_1dz_2.\label{sectorial functional calculus}
\end{equation}
Let $H^{\infty}_{0,1}(\Sigma_{\theta_1}\times \Sigma_{\theta_2})=H^{\infty}_{0}(\Sigma_{\theta_1})\bigoplus H^{\infty}_0(\Sigma_{\theta_2})\bigoplus H^{\infty}_0(\Sigma_{\theta_1}\times \Sigma_{\theta_2})$.
    Hence, for any $f\in H^{\infty}_{0,1}(E_{s_1}\times E_{s_2})$, we can write $f(z_1,z_2)=f_1(z_1)+f_2(z_2)+f_{12}(z_1,z_2)$ for all $(z_1,z_2)\in \Sigma_{\theta_1}\times \Sigma_{\theta_2}$, where $f_1\in H^{\infty}_{0}(\Sigma_{\theta_1}),f_2\in H^{\infty}_{0}(\Sigma_{\theta_2})$ and $f_{12}\in H^{\infty}_0(\Sigma_{\theta_1}\times \Sigma_{\theta_2})$. 
  Let us define $f(A_1,A_2)=f_1(A_1)+f_2(A_2)+f_{12}   (A_1,A_2).$
  
\begin{definition} We say that $(A_1,A_2)$ admits a joint bounded $H^{\infty}( \Sigma_{\theta_1}\times\Sigma_{\theta_2})$- functional calculus if there exists a constant $C > 0$ such that for all $f \in H_{0,1}^{\infty}(\Sigma_{\theta_1}\times\Sigma_{\theta_2})$ we have 
\[
\|f(A_1,A_2)\| \leq C \|f\|_{\infty,  \Sigma_{\theta_1}\times\Sigma_{\theta_2}}.
\]
\end{definition}
 \begin{definition}[Polygonal type operator and Polygonal functional calculus] An operator $T: X\rightarrow X$ is called a polygonal type operator if there exists a convex, open polygon $\Delta\subseteq\mathbb{D}$ such that $\sigma(T)\subseteq \overline\Delta$. Let $T$ be a polygonal type operator, we say T has bounded polygonal functional calculus if there exist a polygon $\Delta$ and a constant $K\geq1$ such that $\lVert \phi(T)\rVert\leq K\lVert\phi\rVert_{\infty,\Delta}$ for all $\phi\in \mathbb{C}[ z ]$.
 \end{definition}
 \begin{definition}($\RE$ Operator) Let $E=\{\xi_1,\xi_2,...,\xi_n\}$ be a subset of $\mathbb{T}$ . Let $T\in B(X)$, we say that $T$ is $\RE$ operator if $ \sigma(T)\subseteq\overline{\mathbb{D}}$ and there is a constant $c>0$ such that 
 \begin{equation*}
     \lVert R(z,T) \rVert \leq c\prod^N_{j=1}\lvert\xi_j-z\rvert^{-1}, \qquad z\in \mathbb{C},1<\lvert z \rvert<2.
 \end{equation*}
 \end{definition}
 \begin{remark}
When $E=\{1\}$ the operator $T$ is called a Ritt operator \cite{MR3293430}.
 We denote the interior of convex hull of $D(0,r)\cup E$ by $E_r$.
 \end{remark}
 \begin{lemma}[\cite{MR4819960}, Lemma 2.8] An operator $T\in B(X)$ is $\RE$ iff there is a $r\in(0,1)$  which is $E$-large enough(in the sense of [\cite{MR4819960},Remark 2.7]) such that  the following two conditions hold
 \begin{enumerate}
     \item[(i)] $\sigma(T)\subseteq \overline E_r$.
     \item[(ii)]  For all $s\in(r,1)$, there exists a constant $c>0$ such that 
     \begin{equation*}
         \lVert R(z,T)\rVert\leq c \prod^N_{j=1}\lvert\xi_j-  z\rvert^{-1}, \qquad z\in D(0,2)\setminus \overline{E_s} .
     \end{equation*}
     
 \end{enumerate}\label{definition of RittE operator}
 \end{lemma}
 We call an operator $T\in B(X)$ is $\RE$ operator of type $r,$ if it satisfies conditions (i) and (ii) of Lemma \ref{definition of RittE operator}.

   For any $s\in(0,1)$ , we denote $H^\infty_0(E_s)$ be the space of all bounded holomorphic functions defined on $E_s$, for which there exist positive reals $c,s_{1},s_{2},s_{3},...,s_{n}>0$ such that $$\lvert f(\lambda)\rvert\leq c          \prod^N_{j=1}\lvert\xi_j-\lambda\rvert^{s_i}, \qquad\lambda \in E_s.$$
   
   Assume that $T$ is a $\RE$ operator of type $r\in(0,1)$, let $s\in(r,1)$ for any $f\in H^{\infty}_{0}(E_s)$, we set 
      \begin{equation}
          f(T)=\frac{1}{2\pi i}\int_ {\partial E_u}f(\lambda)R(\lambda,T)d\lambda,\qquad u\in(r,s).
          \label{functional calculus of Ritt_E operator}
      \end{equation}
      Due to the decay of $f$, the integral is well-defined and absolutely convergent.

      It is easy to check that the map $ f \rightarrow f(T)$ is an algebra homomorphism.
  \begin{definition}(Bounded $H^{\infty}(E_s)$-functional calculus) Let $T$ be a $\RE$ operator type $r\in(0,1)$ and  $s\in(r,1)$ . We say that $T$ admits a bounded $H^{\infty}(E_s)$-functional calculus if there exists a constant $K\geq 1$ such that 
      \begin{equation*}
          \lVert f(T)\rVert\leq K\lVert f\rVert_{\infty,E_s},\qquad \text{for all}\hspace{0.2cm} f \in H^{\infty}_0(E_s).
          \end{equation*}
\end{definition}
We refer to \cite{MR4819960} for more details on $\RE$ operators. We now state  the following two lemmas, whose proof can be written with a similar proof as[\cite{MR4819960}, Lemma 4.2] that will be useful for later purposes.
\begin{lemma}\label{p independent of h}
    Let $A$ be a sectorial operator of type $\omega$ that admits a bounded $H^{\infty}(\Sigma_{\theta})$-functional calculus where $\theta\in(0,\frac{\pi}{2})$. Also consider $h\in H^{\infty}_0(\Sigma_\theta)$. Then, for any $\rho\in(0,1)$, there exists a $K>0$ such that $\lVert h(A_\rho)\rVert\leq K\lVert h\rVert_{\infty,\Sigma_{\theta}}$ where $A_\rho=(1-\rho)+\rho A$.
    \end{lemma}
    \begin{proof}
        Let $\tilde{h}(z)=h(\phi_\rho(z))-\frac{h(1-\rho)}{1+z},$ where $\phi_\rho(z):=(1-\rho)+\rho z.$ One can check that $\tilde{h}\in
H^{\infty}_{0}(\Sigma_\theta).$ Therefore,
$\tilde{h}(A)=\frac{1}{2\pi i}\int_{\partial\Sigma_{u}}\tilde{h}(z)R(z,A)$ is well-defined and $\lVert \tilde{h}(A) \rVert\lesssim\lVert \tilde{h} \rVert_{\infty,\Sigma_{\theta}}$, where $u\in(\omega,\theta)$. From \cite{MR2244037}, we have
$\tilde{h}(A)=\frac{1}{2\pi i}\int_{\partial{\Sigma_u}}\tilde{h}(z)R(z,A)dz=\frac{1}{2\pi i}\int_\Gamma \tilde{h}(z)R(z,A)dz$ where $\Gamma=\partial(\Sigma_u\cup B(0,\delta))$ for some small $\delta>0.$  Let us choose  $x>0$ such that $d(x,\sigma(A))>0$. Consider $\Gamma_1(t)=\delta e^{it}$ where $t\in[-u,u]$. Let $\gamma$ be the closed curve that joins the path $\Gamma_1$ and the path that passes through the points $\delta e^{iu}$, $x+ix\tan u$, $x-ix\tan u$ and $\delta e^{-iu}$ oriented counterclockwise. For a $R>0$, let us define  \begin{align*}W=\frac{1}{2\pi i}\lim\limits_{x\rightarrow\infty}\int^1_{-1} \tilde{h}\left(x+itx \tan u\right)R\left(x+itx\tan u,A\right)ix\tan u\hspace{0.1cm} dt\hspace{1.2cm}\\+\lim\limits_{R\rightarrow\infty}\frac{1}{2\pi i}\int^R_{\infty}\tilde{h}(te^{iu})R(te^{iu},A)dt+\lim\limits_{R\rightarrow\infty}\frac{1}{2\pi i}\int^{\infty}_{R}\tilde{h}(te^{-iu})R(te^{-iu},A)dt.\end{align*} Therefore, we can write
\begin{align*}
    \tilde{h}(A)&=\frac{1}{2\pi i}\int_{\gamma}\tilde{h}(z)R(z,A)dz +W.\\
    &=\frac{1}{2\pi i}\int_{\gamma}{h(\phi_\rho(z))}R(z,A)dz-\frac{1}{2\pi i}\int_{\gamma}\frac{h(1-\rho)}{1+z}R(z,A)dz+W\\
   &= h(A_\rho)-h(1-\rho)(I+A)^{-1}+W    
\end{align*}
In the above, we have used Dunford-Schwartz-Riesz functional calculus.
Therefore, for some $s>0$ we have
\begin{align*}
    &\left\lVert \int^{1}_{-1} \tilde{h}(x+itx \tan u)R(x+ itx\tan u,A)ix\tan u\hspace{0.1cm}dt\right\rVert\\
    &\hspace{1.5cm}\lesssim\int^{1}_{-1} \left\lVert \tilde{h}(x+itx \tan u)R(x+itx\tan u,A)ix\tan u\right\rVert dt
    \\&\hspace{1.5cm}\lesssim\frac{1}{{\lvert x\rvert}^s}\int^1_{-1} \frac{dt}{(1+t^2\tan^2 u)^{\frac{s+1}{2}}}.
 \end{align*}
Since $s>0$, taking limit $x\rightarrow{\infty}$ both sides we will get that
\begin{equation*}
\lim_{x\rightarrow\infty}\int^{1}_{-1}\tilde{h}(x+itx \tan u)R(x+ itx\tan u,A)ix\tan udt=0.
\end{equation*}
One can see that
\begin{align*}
    \lim\limits_{R\rightarrow\infty}\left\lVert \int^{R}_{\infty}\tilde{h}(te^{iu})R(te^{iu},A)dt\right\rVert\lesssim\lim\limits_{R\rightarrow\infty}\int^{\infty}_{R}\frac{dt}{\lvert t\rvert^{s+1}}\to 0.
\end{align*}
Therefore, $$\lim_{R\rightarrow\infty}\int^{R}_{\infty}\tilde{h}(te^{iu})R(te^{iu},A)dt=0.$$
Similarly, one can see that
\begin{equation*}
 \lim\limits_{R\rightarrow\infty}\int^{\infty}_{R}\tilde{h}(te^{-iu})R(te^{-iu}, A)dt=0.  
\end{equation*}
Therefore, we deduce $W=0$.
Therefore, $ \tilde{h}(A)=h(A_\rho)-h(1-\rho)(I+A)^{-1}.$
 As $A$ admits a bouned functional calculus, there exists $K_1>0$ such that \begin{equation}
    \lVert h(A_{\rho})\rVert\leq K_1\lVert \tilde{h}\rVert_{\infty,\Sigma_{\theta}}+\lvert h(1-\rho)\rvert,\qquad \text{for all} \hspace{0.1cm}\rho \in(0,1).
    \end{equation}
From the definition of $\tilde{h}$, we can see that $\lVert\tilde{h}\rVert_{\infty,\Sigma_{\theta}}=\lVert h\rVert_{\infty,\Sigma_{\theta}}.$
Since $(1-\rho) \in \Sigma_{\theta}$, for some $K>0$ we have $\lVert h(A_{\rho})\rVert\leq K\lVert h\rVert_{\infty,\Sigma_{\theta}}$ for all $\rho\in(0,1).$
\end{proof}
\begin{lemma} Let $\mathbf{A_\rho}=(A^\rho_1,A^\rho_2)=\left((1-\rho)I+\rho A_1, (1-\rho)I+\rho A_2\right)$ be a commuting pair of operators such that each $A_i$ is a sectorial operator of type $\omega\in(0,\frac{\pi}{2})$ where $i=1,2$. Also, assume that each $A_i,i=1,2$ has bounded $H^{\infty}(\Sigma_{\theta})-$ functional calculus and $(A_1,A_2)$ has a bounded $H^{\infty}(\Sigma_{\theta}\times \Sigma_{\theta})$-functional calculus.
 Then, there exists a constant $K>0$ such that $\lVert h(\mathbf{A_\rho})\rVert\leq K \lVert h \rVert_{\infty,  \Sigma_{\theta}\times\Sigma_{\theta}}$  for all $\rho\in(0,1)$ and $h\in H^{\infty}_0(\Sigma_{\theta}\times\Sigma_{\theta}).$\label{double h of A-rho}
\end{lemma}
\begin{proof}
    Let $\phi_{\rho}(z_1,z_2)=((1-\rho)+\rho z_1,(1-\rho)+\rho z_2).$
    Let us consider the function \begin{equation}
        \tilde{h}(z_1,z_2)=h(\phi_\rho(z_1,z_2))-\frac{h(\phi_{\rho}(z_1,0)}{1+z_2}-\frac{h\left(\phi(0,z_2)\right)}{1+z_1}-\frac{h(1-\rho,1-\rho)}{(1+z_1)(1+z_2)}.\label{h-tilde}
    \end{equation}
Note that $\tilde{h}\in H^{\infty}_0(\Sigma_\theta\times\Sigma_\theta).$ Let $d(x,A_i)>0,i=1,2$ and the path $\gamma$ be defined as in Lemma \ref{p independent of h}.
Let $\tilde{\gamma}$ is a path starting from $+\infty$ then going to $x+ix\tan u$ and $x-ix\tan u$ at last ended at $-\infty$.

Consider the following
\begin{align*}
&m_1=\frac{1}{(2\pi i)^2}\int_{\gamma}\int_{\gamma}\tilde{h}(z_1,z_2)R(z_1,A_1)R(z_2,A_2)dz_1dz_2\\
&m_2=\frac{1}{(2\pi i)^2}\int_{\gamma}\int_{\tilde{\gamma}}\tilde{h}(z_1,z_2)R(z_1,A_1)R(z_2,A_2)dz_1dz_2\\
&m_3=\frac{1}{(2\pi i)^2}\int_{\tilde{\gamma}}\int_{\gamma}\tilde{h}(z_1,z_2)R(z_1,A_1)R(z_2,A_2)dz_1dz_2\\
&m_4=\frac{1}{(2\pi i)^2}\int_{\tilde{\gamma}}\int_{\tilde{\gamma}}\tilde{h}(z_1,z_2)R(z_1,A_1)R(z_2,A_2)dz_1dz_2.
\end{align*}
Therefore, 
 \begin{align*}
\tilde{h}(A_1,A_2)&=\frac{1}{(2\pi i)^2}\int_{\partial\Sigma_{\theta}}\int_{\partial\Sigma_{\theta}}\tilde{h}(z_1,z_2)R(z_1,A_1)R(z_2,A_2)dz_1dz_2\\&=\lim_{x\rightarrow
\infty} (m_1+m_2+m_3+m_4).
\end{align*}
Let us consider \begin{equation*}
    M_1(x)=\int_{\gamma}\frac{d\lvert z_1\rvert}{\lvert z_1\rvert^{s+1}}, M_2(x)=\int_{\tilde{\gamma}}\frac{d\lvert z_2\rvert}{\lvert z_2\rvert^{s+1}}.
\end{equation*}
Notice that we have \begin{align*}
     \int_{\gamma}\frac{d\lvert z_1\rvert}{\lvert z_1\rvert^{s+1}}&=\int^{u}_{-u} \frac{\delta dt}{\delta^{s+1}} +2\int^{\frac{x}{cosu}}_{\delta}\frac{dt}{t^{s+1}}+\int^{1}_{-1}\frac{\lvert x\tan u\rvert dt}{\lvert x+ixt\tan u\rvert^{s+1}}
 \end{align*}
 One can check that \begin{align*}
     \lim_{x\rightarrow\infty}\int_{\gamma}\frac{d\lvert z_1\rvert}{\lvert z_1\rvert^{s+1}}= m, \text{where}\hspace{0.1cm} m>0.
     \end{align*}
Moreover, we see that
\begin{align*}
    \lim\limits_{x\rightarrow\infty}M_2(x)&=\lim\limits_{x\rightarrow\infty}\int^1_{-1}\frac{\lvert x\tan u\rvert dt_2}{\lvert x+ix t_2 \tan u\rvert^{s+1}}+\lim\limits_{R\rightarrow\infty}\int^{R}_{\infty}\frac{dt_2}{\lvert t_2e^{iu}\rvert^{s+1}}+\lim\limits_{R\rightarrow\infty}\int^{\infty}_{R}\frac{dt_2}{\lvert t_2e^{-iu}\rvert^{s+1}}\\
    &=0.
\end{align*}
Next observe that
   \begin{align*} &\lim_{x\rightarrow\infty}\left\lVert\frac{1}{(2\pi i)^2}\int_{\gamma}\int_{\tilde{\gamma}}\tilde{h}(z_1,z_2)R(z_1,A_1)R(z_2,A_2)dz_1dz_2\right\rVert\\&\lesssim\lim_{x\rightarrow\infty}\int_{\gamma}\int_{\tilde{\gamma}}\left\lVert\tilde{h}(z_1,z_2)R(z_1,A_1)R(z_2,A_2)\right\rVert d\lvert z_1 \rvert d\lvert z_2\rvert
    \\&\lesssim \lim_{x\rightarrow\infty}\int_{\gamma}\frac{d\lvert z_1\rvert}{\lvert z_1\rvert^{s+1}}\int_{\tilde{\gamma}}\frac{d\lvert z_2\rvert}{\lvert z_2\rvert^{s+1}}
    \\&\lesssim\lim_{x\rightarrow\infty} M_1(x)\lim_{x\rightarrow\infty}  M_2(x)=0.
 \end{align*}
 Therefore, $\lim\limits_{x\rightarrow\infty}m_2=0.$
 Similarly one can check that $\lim\limits_{x\rightarrow\infty}m_3=0.$ Let us consider $v_j=x + i\hspace{0.1cm}t_j x\tan{u},j= 1,2.$
Let us define the following
 \begin{align*}
&w_1=\lim\limits_{R\rightarrow\infty}\frac{1}{2\pi i}\int^R_{\infty}\int^R_{\infty}\tilde{h}(t_1e^{iu},t_2e^{iu})R(t_1e^{iu},A_1)R(t_2e^{iu},A_2) e^{i2u}dt_1dt_2\\
&w_2=\lim\limits_{R\rightarrow\infty}\frac{1}{2\pi i}\int^{\infty}_{R}\int^{\infty}_R\tilde{h}(t_1e^{-iu},t_2e^{-iu})R(t_1e^{-iu},A_1)R(t_2e^{-iu},A_2)e^{-i2u} dt_1dt_2\\ 
&w_3=\lim_{x\rightarrow\infty}\frac{1}{(2\pi i)^2}\int^{1}_{-1}\int^{1}_{-1}\tilde{h}\left(v_1,v_2\right)R\left(v_1,A_1\right)R\left(v_2,A_2\right)(-x^2 \tan^2 u) dt_1dt_2.
\end{align*}
 Note that \begin{align*}
     &\left\lVert\int^{1}_{-1} \int^{1}_{-1} \tilde{h}\left(v_1,v_2\right)R(v_1,A_1)R(v_2,A_2)4x\tan u\hspace{0.1cm}dt_1 dt_2\right\rVert\\&\hspace{1 cm}\lesssim\int^1_{-1}\int^1_{-1}\frac{ x^2\tan^2 u \hspace{0.1cm}dt_1dt_2}{(x^2+4t^2_1x^2\tan^2 u)^\frac{s+1}{2}(x^2+4t^2_2x^2\tan^2 u)^\frac{s+1}{2}}
     \\&\hspace{1cm}\lesssim \frac{1}{\lvert x \rvert^{2s}}.
\end{align*}
 Taking $x\rightarrow\infty$, we deduce that
 $w_3=0.$
 One can check that $w_1=0$ and $w_2=0$.
 Therefore, $m_4=w_1+w_2+w_3=0.$
So, we can write
 \begin{align*}
 \tilde{h}(A_1,A_2)&=\frac{1}{(2\pi i)^2}\int_{\partial\Sigma_{\theta}}\int_{\partial\Sigma_{\theta}}\tilde{h}(z_1,z_2)R(z_1,A_1)R(z_2,A_2)dz_1dz_2\\&=\frac{1}{(2\pi i)^2}\lim\limits_{x\rightarrow\infty}\int_{\gamma}\int_{\gamma}\tilde{h}(z_1,z_2)R(z_1,A_1)R(z_2,A_2)dz_1dz_2.
\end{align*}
Putting the value of $\tilde{h}$ from the equation \ref{h-tilde}, we have 
\begin{align}
 h(\mathbf{A_\rho})=\tilde{h}(A_1,A_2)+h\left(\phi_{\rho}(A_1,0)\right)(1+A_2)^{-1}+h\left(\phi_{\rho}(0,A_2)\right)(I+A_1)^{-1}\label{A-p for h}\end{align}\begin{align*}\hspace{2.5cm}+h\left((1-\rho),(1-\rho)\right)(1+A_1)^{-1}(1+A_2)^{-1}.
\end{align*}
 From Lemma \ref{p independent of h}, there exist $K_1,K_2>0$ such that 
\begin{equation}
     \lVert h(\phi_{\rho}(A_1,0))\rVert\leq K_1\lVert h \rVert_{\infty,\Sigma_{\theta}\times\Sigma_{\theta}}\label{A-1p}
 \end{equation}
\begin{equation}
     \lVert h\left(\phi_{\rho}(0,A_2)\right)\rVert\leq K_2\lVert h \rVert_{\infty,\Sigma_{\theta}\times\Sigma_{\theta}}.\label{A-2p}
 \end{equation}
  Combining inequalities in \ref{A-1p} and \ref{A-2p} with the equation \ref{A-p for h},  for any $\rho\in(0,1)$ we have
 \begin{equation*}
    \lVert\tilde{h}(\mathbf{A_\rho})\rVert\leq C_1\lVert h \rVert_{\infty,\Sigma_{\theta}\times\Sigma_{\theta}}+ C_2\lvert h(1-\rho,1-\rho)\rvert,\hspace{0.1 cm} \text{for some} \hspace{0.1cm} C_1,C_2>0.   
 \end{equation*}
 Since $1-\rho\in\Sigma_{\theta}$, it follows that $\lVert h(\mathbf{A_\rho})\rvert\leq K \lVert h \rVert_{\infty,\Sigma_{\theta}\times\Sigma_{\theta}}$ for some $K>0$ which is independent of $\rho$.
\end{proof}
 \begin{definition} Let $1<p\neq 2<\infty.$ Suppose $(T_1,T_2)$ be a commuting pair of bounded operators on $L^p(\Omega)$. We say that  $(T_1,T_2)$ admits a joint isometric loose
dilation, if there exists a measure space $\Omega^\prime,$ a commuting tuple of onto isometries $\mathbf{U}=(U_1,U_2),$ on $L^p(\Omega^\prime),$ 

together with two bounded operators $\mathcal{Q}
:L^p(\Omega^\prime)\to L^p(\Omega)$ and $\mathcal{J}:L^p(\Omega)\to L^p(\Omega^\prime),$ such that $$T_1^{i_1}T_2^{i_2}=\mathcal{Q}U_1^{i_1}U_2^{i_2}\mathcal{J}$$ for all $i_1,i_2\in\mathbb{N}_0.$ 

\end{definition}
The concept of loose dilation leads to the notion of $p$-polynomially boundedness and $p$-completely polynomially boundedness (see \cite{MR3265289}). We need the following definitions.

Let us define the shift operators on $\ell^p(\mathbb{Z}\times\mathbb{Z})$ as 
\[
              (S_1(a))_{i_1,i_2}:=a_{i_1-1,i_2},\ a\in\ell^p(\mathbb{Z}\times\mathbb{Z}).
              \]
    \[
              (S_2(a))_{i_1,i_2}:=a_{i_1,i_2-1},\ a\in\ell^p(\mathbb{Z}\times\mathbb{Z}).
              \]

\begin{definition}[Jointly $p$-polynomially bounded] Let $1<p<\infty.$ Let $(T_1,T_2)$ be a commuting tuple of bounded operators on $X$. We say that $(T_1,T_2)$ is jointly 
$p$-polynomially bounded, if there exists a $K>0$, such that
\begin{equation*}
\|\phi(T_1,T_2)\|_{X\to X}\leq K\|\phi(\mathcal{S})\|_{\ell^p(\mathbb{Z}\times\mathbb{Z})\to\ell^p(\mathbb{Z}\times\mathbb{Z})},
\end{equation*}
for any polynomial $\phi\in\mathbb{C}[z_1,z_2],$ where  $\mathcal{S}:=(S_1,S_2)$ is the commuting tuple of left shift operators on $\ell^p(\mathbb{Z}\times\mathbb{Z}).$
\end{definition}
\begin{definition}[Jointly $p$-completely polynomially bounded]Let $1<p<\infty$ and  $(T_1,T_2)$ be a commuting tuple of bounded operators on $X.$ We say that $(T_1,T_2)$ is jointly $p$-completely polynomially bounded, there exists a constant $C>0,$ such that \[\big\|(\phi_{i,j}(T_1,T_2))\big\|_{L^p([2],X)\to L^p([2],X)}\leq C\big\|(\phi_{i,j}(\mathcal{S}))\big\|_{L^p([2],\ell^p(\mathbb{Z}\times\mathbb{Z}))\to L^p([2],\ell^p(\mathbb{Z}\times\mathbb{Z})},\]
where $\phi_{i,j}\in\mathbb{C}[z_1,z_2]$.
\end{definition}
\section{Joint functional calculus of \texorpdfstring{$\RE$}{RE} Operators.} \label{jointritte}       
   We can generalize the similar idea to define bounded $H^{\infty}$-functional calculus for the commuting family of polygonal and $\RE$ operators.
It suffices to study for a pair of operators; the general case follows similarly.
    Let $E\subseteq\mathbb{T}$ as defined in \ref{transferprin}, for $s_1,s_2 \in (0,1)$ we  define  the space $H^{\infty}_0(E_{s_1}\times E_{s_2})$ as the set of all bounded holomorphic function define on $E_{s_1}\times E_{s_2}$, for which there exists positive reals $c,s_1,s_2,...,s_n,s'_1,s'_2,....,s'_n>0$ such that 
    \begin{equation}
   \lvert f(\lambda_1,\lambda_2)\rvert\leq c\prod^N_{j=1}\lvert\xi_j-\lambda_1\rvert^{s_i}\prod^N_{j=1}\lvert\xi_j-\lambda_2\rvert^{s'_i}, \qquad (\lambda_1,\lambda_2)\in E_s\times E_s.
    \end{equation}
     Let $(T_1,T_2)$ be a commuting tuple of operators such that each $T_i$ is $\RE$ operator of type $r_i$ for $i=1,2$.
    Let $f \in H^{\infty}_0(E_{s_1}\times E_{s_2})$ and $u_i\in(r_i,s_i)$ for $i=1,2$, define 
    \begin{equation}
         f(T_1,T_2):= \int_{\partial E_{u_1}}\int_{\partial E_{u_2}}f(\lambda_1,\lambda_2)R(\lambda_1,T_1)R(\lambda_2,T_2)d\lambda_1d\lambda_2\qquad
          \text{for all $\lambda_1,\lambda_2\in E_s$}. 
          \label{double ritt_E operator}
     \end{equation}
    It is easy to check that $f(T_1,T_2)$ is well-defined and the integral in \ref{double ritt_E operator} is absolutely convergent.
    
   Let $H^{\infty}_{0,1}(E_{s_1}\times E_{s_2})=H^{\infty}_{0}(E_{s_1})\bigoplus H^{\infty}_0(E_{s_2})\bigoplus H^{\infty}_0(E_{s_1}\times E_{s_2})$.
    Hence $f\in H^{\infty}_{0,1}(E_{s_1}\times E_{s_2})$ , we can write $f(z_1,z_2)=f_1(z_1)+f_2(z_2)+f_{12}(z_1,z_2)$ for all $(z_1,z_2)\in E_{s_1}\times E_{s_2}$, where $f_1\in H^{\infty}_{0}(E_{s_1}),f_2\in H^{\infty}_{0}(E_{s_2})$ and $f_{12}\in H^{\infty}_0(E_{s_1}\times E_{s_2})$.

  Consider $f(T_1,T_2)=f_1(T_1)+f_2(T_2)+f_{12}   (T_1,T_2)$ where $f_1(T_1),f_2(T_2)$ is determined as equation \ref{functional calculus of Ritt_E operator} and $f_{12}(T_1,T_2)$ is determined as equation \ref{double ritt_E operator}.
\begin{definition}
       Let $s_i\in(r_i,1)$, $i=1,2$. We say that $(T_1,T_2)$ admits a bounded $H^{\infty}(E_{s_1}\times E_{s_2})$-functional calculus if
  \begin{equation*}
      \lVert f(T_1,T_2)\rVert\lesssim\lVert f\rVert_{\infty,E_{s_1}\times E_{s_2}}, \qquad \forall  f\in H^{\infty}_{0,1}(E_{s_1}\times E_{s_2}). 
  \end{equation*}
  \end{definition}
\begin{definition}[Joint Polygonally bounded operator] A  commuting tuple $(T_1,T_2)$ on a Banach space $X$ is said to be joint polygonally bounded if there exists a convex, open polygon $\Delta\subseteq\mathbb{D}$ and $\sigma(T_1),\sigma(T_2)\subseteq \overline\Delta$  such that $\lVert\phi(T_1, T_2)\rVert\lesssim \lVert\phi\rVert_{\infty,\Delta\times\Delta}$ for all $\phi\in\mathbb{C}[z_1,z_2]$.  
\end{definition}
\begin{proposition}\label{polygonally bounded}
      Let $(T_1,T_2)$ be the commuting pair $\RE$ operators of type $r$. Then $(T_1,T_2)$ admits a bounded $H^{\infty}(E_{s}\times E_{s})$-functional calculus iff 
    \begin{align*}
        \lVert \phi(T_1,T_2)\rVert\lesssim\lVert\phi\rVert_{\infty,E_s\times E_s} , \qquad \forall \phi\in\mathbb{C}[z_1,z_2].
    \end{align*}
\end{proposition}
\begin{proof}
Let $\phi \in\mathbb{C}[z_1,z_2].$ We can write $\phi(z_1,z_2)=\sum\limits^{d}_{n=1}a_n(z_1)b_n(z_2)$.
In addition, we can write $a_n(z_1)=a^0_n(z_1)+a^1_n(z_1)$ and $ b_n(z_2)=b^0_n(z_2)+b^1_n(z_2)$,
such that $a^0_n=\sum\limits^N_{j=1}a_n(\xi_j)L_j$, $b^0_n=\sum\limits^N_{j=1}b_n(\xi_j)L_j$ and $a^1_n(z_1)=a_n(z_1)-a^0_n(z_1)$, $b^1_n(z_2)=b_n(z_2)-b^0_n(z_2)$  where $L
_1,L_2,...,L_N \in \mathbb{C}[z]$ satisfies $L_i(\xi_j)=\delta_{ij}$ for all $1\leq i,j\leq N $.
 The polynomial $\phi$ can be written as 
 \begin{equation}
     \phi(z_1,z_2)=\sum\limits^d_{n=1}(a^0_n(z_1)b^0_n(z_2)+a^0_n(z_1)b^1_n(z_2)+a^1_n(z_1)b^0_n(z_2)+a^1_n(z_1)b^1_n(z_2))
\end{equation}
Therefore, we obtain the following estimate.
     \begin{align*}        
       \Big\|\sum \limits^d_{n=1}a^0_n(T_1)b^0_n(T_2)\Big\|&=\Big\lVert\sum\limits^{d}_{n=1}\{\sum\limits^N_{j=1}a_n(\xi_j)L_j(T_1))(\sum\limits^N_{k=1}b_n(\xi_j)L_k(T_2)\}\Big\rVert\\
         &\leq\sum\limits^N_{j,k=1}\Big\lVert L_jL_k \Big\rVert\Big\lVert \sum^d_{n=1}a_n(\xi_j)b_n(\xi_k)\Big\rVert \\
&\lesssim\sum\limits^N_{j,k=1}\lVert\phi(\xi_j,\xi_k)\rVert\lesssim\lVert\phi\rVert_{\infty,E_s\times E_s}.
        \end{align*}
        Moreover, we also have the following
    \begin{align*}
       \Big\lVert \sum^d_{n=1}a^0_n(T_1)b^1_n(T_2)\Big\rVert
       &\lesssim\sum^N_{j=1}\Big\lVert\sum\limits^d_{n=1} a_n(\xi_j) b^1_n(T_2)\Big\rVert\\ &\lesssim\sum\limits^N_{j=1}\Big\lVert\sum^d_{n=1}( a_n(\xi_j)b^1_n)\Big\rVert_{\infty,E_s}\\&\lesssim \sum^N_{j=1}\lVert\phi(\xi_j,.)\rVert_{\infty,E_s}\\
       &\lesssim \lVert\phi\rVert_{\infty,E_s\times E_s}. 
    \end{align*}
    Similarly one can be able to show that \[\Big\lVert \sum\limits^d_{n=1}a^1_n(T_1)b^0_n(T_2)\Big\rVert\lesssim\lVert \phi \rVert_{\infty,E_s\times E_s}.\] Observe that
$\sum\limits^d_{n=1}a^1_n(\xi_i)b^1_n(\xi_j)=0.$ Therefore, $\sum\limits^d_{n=1}a^1_n(z_1)b^1_n(z_2)\in H^{\infty}_0(E_s\times E_s).$
Let us define the following polynomials
\begin{align*}
&P_1(z_1,z_2)=\sum\limits^d_{n=1}a^0_n(z_1)b^0_n(z_2)\\
&P_2(z_1,z_2)=\sum\limits^d_{n=1}a^0_n(z_1)b^1_n(z_2)\\
&P_3(z_1,z_2)=\sum\limits^d_{n=1}a^1_n(z_1)b^0_n(z_2)\\
&P_4(z_1,z_2)=\sum\limits^d_{n=1}a^1_n(z_1)b^1_n(z_2).
\end{align*}
Therefore, 
\begin{align}\label{equation-P-4}
    \lVert P_4 \rVert_{\infty,E_s\times E_s}\leq \lVert P_1\rVert_{{\infty,E_s\times E_s}}+\lVert P_2\rVert_{\infty,E_s\times E_s}+\lVert P_3\rVert_{{\infty,E_s\times E_s}}+\lVert \phi\rVert_{\infty,E_s\times E_s }.
\end{align}
Further, we have
\begin{align*}
\lVert P_1\rVert_{\infty,E_s\times E_s}\lesssim\sum\limits^N_{j,k=1}\Big\lVert L_jL_k \Big\rVert\Big\lVert \sum^d_{n=1}a_n(\xi_j)b_n(\xi_k)\Big\rVert
\lesssim\sum^N_{j,k=1}\lvert \phi(\xi_j,\xi_k)\rvert\lesssim\lVert\phi\rVert_{\infty,E_s\times E_s}.
\end{align*}
Also, we deduce that
\begin{align*}
    \lVert P_2\rVert_{\infty,E_s\times E_s}&\lesssim\sum\limits^N_{j=1}\Big\lVert\sum^d_{n=1}( a_n(\xi_j)b^1_n)\Big\rVert_{\infty,E_s}\\&\lesssim \sum^N_{j=1}\lVert\phi(\xi_j,.)\rVert_{\infty,E_s}\\
       &\lesssim \lVert\phi\rVert_{\infty,E_s\times E_s}.
\end{align*}
Similarly, the following estimate hold 
\begin{align*}
    \lVert P_3 \rVert_{\infty, E_s\times E_s}\lesssim\lVert\phi\rVert_{\infty,E_s\times E_s}
\end{align*}
Therefore, Using equation\ref{equation-P-4}, we obtain that 
$\lVert P_4\rVert_{\infty,E_s\times E_s}\lesssim\lVert \phi\rVert_{\infty,E_s\times E_s}$
By our assumption, we have
\begin{align*}
    \lVert P_4(T_1,T_2)\rVert\lesssim\lVert P_4\rVert_{\infty,E_s\times E_s}.
\end{align*}
Therefore,
\begin{align*}
    \lVert P_4(T_1,T_2)\rVert\lesssim \lVert \phi\rVert_{\infty,E_s\times E_s}.
\end{align*}

Combining all together, it is immediate that \begin{equation*}
    \lVert \phi(T_1,T_2)\rVert\lesssim\lVert\phi\rVert_{\infty,E_s \times E_s},\qquad \text{for all}\hspace{0.2cm}\phi\in\mathbb{C}[z_1,z_2].
    \end{equation*}

Conversely, let  $\lVert\phi(T_1,T_2)\rVert\lesssim\lVert\phi\rVert_{\infty,E_s\times E_s}$ for all $\phi\in\mathbb{C}[z_1,z_2]$. Also, assume that $f\in H^{\infty}(E_s\times E_s)$ and $t\in(0,1)$ and $t'\in(t,1)$. Therefore, by [\cite{MR4043874}, Lemma 2.4], there exists a sequence polynomial $(\phi_m)_m$ of $\mathbb{C}[z_1,z_2]$ such that $\phi_m\rightarrow f$ uniformly in $t'\overline{E_s}\times t'\overline{E_s}$ as $m\rightarrow \infty.$ 
Since $\sigma(tT_1),\sigma(tT_2)\subset t'E_u$, by classical Dunford-Riesz functional calculus, we have that
 \[\phi_m(tT_1,tT_2)=(\frac{1}{2\pi i})^2\int_ {\partial t'E_u} \int_{\partial t'E_u} \phi_m(\lambda_1,\lambda_2)R(\lambda_1,tT_1)R(\lambda_2,tT_2)d\lambda_1d\lambda_2. \]
Note that as $m\rightarrow{\infty}$ , $\phi_m\rightarrow f$. Hence, $\lVert\phi_m\rVert_{\infty},_{t'E_s\times t'E_s}\rightarrow \lVert f \rVert_{\infty,t'E_s\times t'E_s}$.

From our assumption, we have that
 \begin{equation}
    \lVert \phi_m(tT_1,tT_2)\rVert\lesssim \lVert \phi_m\rVert_{{\infty},t'E_s\times t'E_s}.
\end{equation}
    Now taking $m\rightarrow{\infty}$ both sides, we have
    \begin{equation}
        \lVert f(tT_1,tT_2)\rVert\lesssim  \lVert f\rVert_{\infty,t'E_s\times t'E_s}.
    \end{equation}
By the Dominated Convergence Theorem, we get that 
         $\lim\limits_{t\rightarrow 1}f(tT_1,tT_2)=f(T_1,T_2)$.
         
Combining with (10), it follows that
         \begin{equation}
             \lVert f(T_1,T_2)\rVert\lesssim \lVert f \rVert_{\infty,t'E_s\times t'E_s}\lesssim \lVert f\rVert_{\infty,E_s\times E_s} , \qquad   \forall f\in H^{\infty}_0(E_s\times E_s).
         \end{equation}
     Hence, the proof of the proposition is completed. \end{proof}
     \section{Transfer principle: Proof of Theorem \eqref{transferprin}}\label{sectran}
The proof of Theorem \ref{transferprin} is long and will be divided into several lemmas. We also use the notation as in Sections \ref{prelim} and \ref{jointritte}. The proof for $N=1$ is rather straightforward and can be done along the lines of \cite{MR3928691}. Hence, we assume that $N\geq2$. Let us define $A^1_i=I_{X}-\overline{\xi}_iT_1$ and $A^2_i=I_{X}-\overline{\xi_i}T_2$ , where $i=1,2\dots,N$. Let $(A^1_i,A^2_j)$ has bounded $H^{\infty}(\Sigma_{\theta_{i}}\times\Sigma_{\theta_{j}})$- functional calculus. We may choose $\theta\in(0,\frac{\pi}{2})$ close enough to $\frac{\pi}{2}$ so that  $(A^1_{i},A^2_j)$ admits bounded $H^{\infty}(\Sigma_{\theta}\times \Sigma_{\theta})$-functional calculus for all
    $i,j\in\{1,2,\dots,N\}.$ Assume that $\{\xi_1,\xi_2,\dots,\xi_N\}$ is oriented counter clockwise on $\mathbb{T}$ and define $\xi_{N+1}=\xi_1$. 
According to the proof of [\cite{MR4819960}, Theorem 4.3], we can find a closed arc $\Gamma$ of $\mathbb{C}$ joining the point where $\partial\Sigma(\xi_i,\theta)_+$ and $\partial\Sigma(\xi_i,\theta)_-$ meet on $\mathbb{T}$. For a $r\in(0,1)$ and $\theta'\in(0,\frac{\pi}{2})$ we can obtain a set with the points $\{z_1,z_2,\dots,z_n\}$ on $r\Gamma$ ordered counter clockwise, such that 
\begin{enumerate}
    \item [(i)] $\sigma(T)\subset\Sigma(z_i,\theta')$, $i=1,\dots,p.$
    \item[(ii)] $\partial\Sigma(\xi_i,\theta)_+$ and $\partial\Sigma(\xi_i,\theta)_-$ meets in $\mathbb{D}\setminus\{0\}.$
    \item[(iii)]$\partial\Sigma(z_i,\theta')_+$ and $\partial\Sigma(z_i,\theta')_-$ meets in $\mathbb{D}\setminus\{0\}$, $i=1, \dots, p-1$.
\end{enumerate}
Putting together the point $\xi_i$ and $z_i$, we obtain the set $\{\zeta_1,\dots,\zeta_m\}$ of distinct points in $\mathbb{D}\setminus\{0\}$ oriented counter clockwise with the angle $\beta_i\in(0,\frac{\pi}{2})$, $i=1,\dots,m$; such that
\begin{enumerate}
    \item [(i)] $\{\zeta_1,\dots,\zeta_N\}\cap\mathbb{T}=E$
    \item [(ii)]For any $i=1,\dots,m$ , if there exist $j\in\{1,\dots,N\}$ such that $\zeta_i=\xi_j$, then $\beta_i=\theta.$
    \item[(iii)] define $\zeta_{m+1}=\zeta_1$, the half-line $\partial\Sigma(\zeta_i,\beta_i)_+$ and $\partial\Sigma(\zeta_i,\beta_i)_-$ meets exactly at one point $c_i\in\mathbb{D}\setminus\{0\},$ for all $i=1,2,\dots,m.$
    \end{enumerate}
    Now, consider $d_i=\frac{1}{2}\Big(c_i+\frac{c_i}{\lvert c_i\rvert}\Big)$, $i=1,2,\dots,m.$
    In this way we can contract two open convex polygons $\Delta_0$ and $\Delta$ with vertex $\{\zeta_1,c_1,\dots,\zeta_m,c_m\}$ and $\{\zeta_1,d_1,\dots,\zeta_m,d_m\}$ respectively.
    Observe that $\Delta_0\subset\Delta.$

 Let $\phi\in\mathbb{C}[z_1,z_2]$ such that $\phi(0,0)=0$. Consider the functions $\phi_{ij}:(\mathbb{C}\setminus\gamma_i)\times (\mathbb{C}\setminus\gamma_j)\rightarrow \mathbb{C}$, where $1\leq i,j\leq m$ defined as
    \begin{equation} 
    \phi_{ij}(z_1,z_2)=(\frac{1}{2\pi i})^2 \int_{\gamma_i}\int_{\gamma_j}\frac{\phi(\lambda_1,\lambda_2)}{(\lambda_1-z_1)(\lambda_2-z_2)}d\lambda_1 d\lambda_2 .\label{phi ij}
    \end{equation}
 So, by Cauchy's theorem, we have that 
          \begin{equation}
              \phi(z_1,z_2)=\sum\limits_{1\leq i,j\leq m}\phi_{ij}(z_1,z_2)\label{cauchy}.\end{equation}
    Let $i,j\in\{1,\dots,m\}$ and assume first that there exists $p,q\in\{1,\dots,N\}$ such that $\zeta_i=\xi_p$ and $\zeta_j=\xi_q$.
Consider the following functions \begin{equation}    
g_{ij}(z_1,z_2) = \phi_{ij}(\xi_p(1-z_1),\xi_q(1-z_2)).\label{g_{ij}} 
\end{equation}
Therefore,
\begin{equation}
    g_{ij}(z_1,z_2)=\frac{1}{(2\pi i)^2}\int_{\gamma_i}\int_{\gamma_j}\frac{\phi(\lambda_1,\lambda_2)}{(\lambda_1-\xi_p(1-z_1))(\lambda_2-\xi_q(1-z_2))}d\lambda_1 d\lambda_2.
    \end{equation}
For $i,j\in\{1,2,\dots,N\}$, let us define the sets \[\Delta_{i}=\{z\in \mathbb{C}: z=\xi_p(1-z') \ \ \    \text{where} \ \ z'\in \Delta\}\]
\[\Delta_{j}=\{z\in \mathbb{C}: z=\xi_q(1-z') \ \ \    \text{where} \ \ z'\in \Delta\}\]
\[\Delta^0_{i}=\{z\in \mathbb{C}: z=\xi_p(1-z') \ \ \    \text{where} \ \ z'\in \Delta^0\}\]\[\Delta^0_{j}=\{z\in \mathbb{C}: z=\xi_q(1-z') \ \ \    \text{where} \ \ z'\in \Delta^0\}\]
\begin{figure}[h]
    \centering
    \includegraphics[width=1\linewidth]{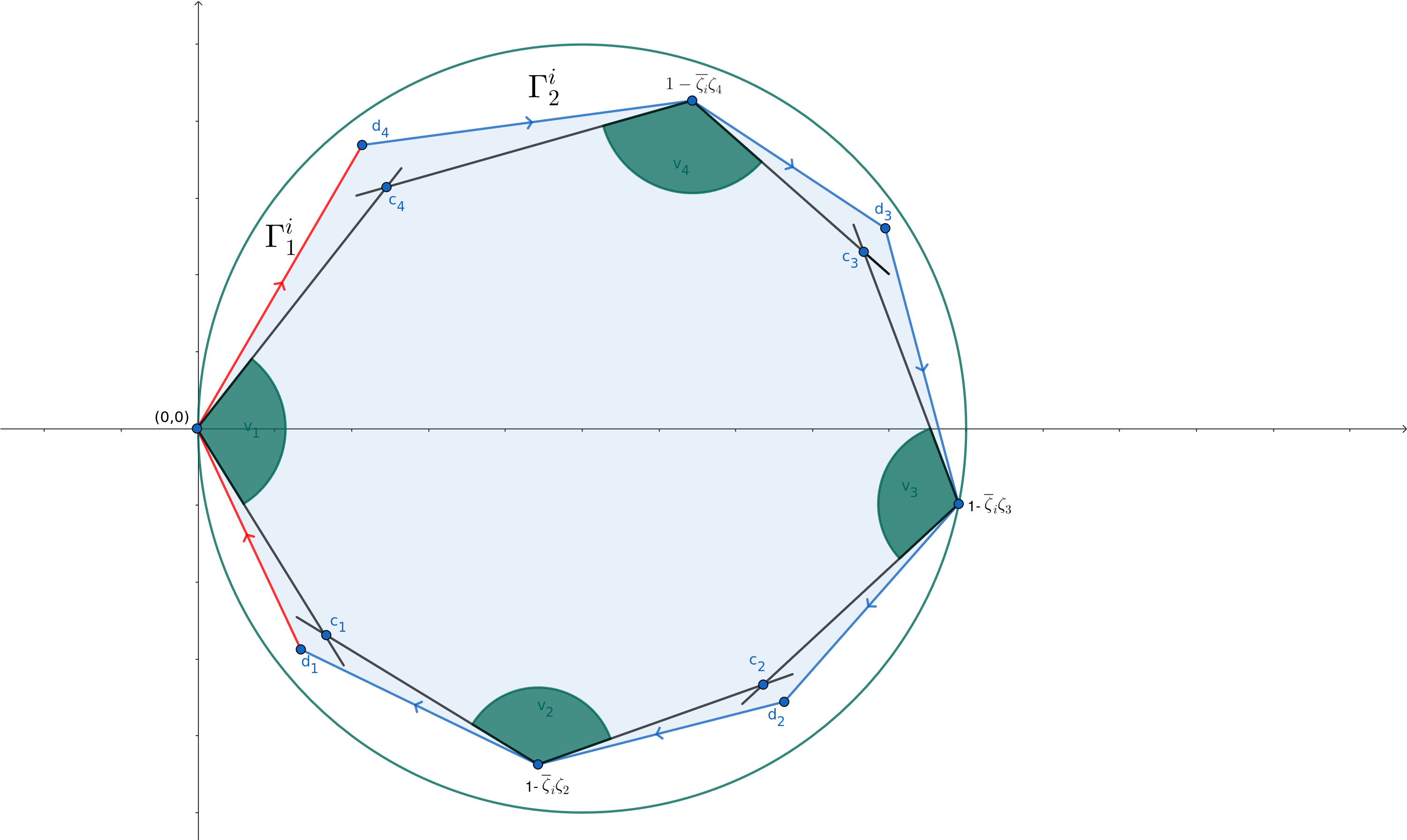}
    \caption{polygon $\Delta^0_{i}$ and $ \Delta_{i}$.}
    \label{fig:enter-label}
\end{figure}

    Let $\Gamma^i_1(t)=1-\overline{\xi}_p\gamma_i(t)$ and $\Gamma^j_1(t)=1-\overline{\xi}_q\gamma_j(t)$. Let us define $\Gamma^i_2(t)=\left(\partial\Delta\setminus\Gamma^i_1\right)(t)$ and $\Gamma^j_2(t)=\left(\partial\Delta\setminus\Gamma^j_1\right)(t)$.
We denote $\psi(z_1,z_2)=\phi(\xi_p(1-z_1),\xi_q(1-z_2))$. Consider the following functions
\begin{align*}
    F_{11}(z_1,z_2)=\frac{1}{(2\pi i)^2}\int_{\Gamma^i_1}\int_{\Gamma^j_1}\frac{\psi(\lambda_1,\lambda_2)}{(\lambda_1-z_1)(\lambda_2-z_2)}d\lambda_1d\lambda_2,\qquad (z_1,z_2)\in(\mathbb{C}\setminus\Gamma^i_1)\times(\mathbb{C}\setminus\Gamma^j_1).
\end{align*}
\begin{align*}
    F_{12}(z_1,z_2)=\frac{1}{(2\pi i)^2}\int_{\Gamma^i_1}\int_{\Gamma^j_2}\frac{\psi(\lambda_1,\lambda_2)}{(\lambda_1-z_1)(\lambda_2-z_2)}d\lambda_1d\lambda_2,\qquad (z_1,z_2)\in(\mathbb{C}\setminus\Gamma^i_1)\times(\mathbb{C}\setminus\Gamma^j_2).
\end{align*}
\begin{align*}
    F_{21}(z_1,z_2)=\frac{1}{(2\pi i)^2}\int_{\Gamma^i_2}\int_{\Gamma^j_1}\frac{\psi(\lambda_1,\lambda_2)}{(\lambda_1-z_1)(\lambda_2-z_2)}d\lambda_1d\lambda_2,\qquad(z_1,z_2)\in(\mathbb{C}\setminus\Gamma^i_1)\times(\mathbb{C}\setminus\Gamma^j_1).
\end{align*}
\begin{align*}
     F_{22}(z_1,z_2)=\frac{1}{(2\pi i)^2}\int_{\Gamma^i_2}\int_{\Gamma^j_2}\frac{\psi(\lambda_1,\lambda_2)}{(\lambda_1-z_1)(\lambda_2-z_2)}d\lambda_1d\lambda_2,\qquad (z_1,z_2)\in(\mathbb{C}\setminus\Gamma^i_2)\times(\mathbb{C}\setminus\Gamma^j_2).
     \end{align*}
By Cauchy's theorem, we have that
\begin{equation}
\psi(z_1,z_2)=F_{11}(z_1,z_2)+F_{12}(z_1,z_2)+F_{21}(z_1,z_2)+F_{22}(z_1,z_2).
\end{equation} 
Also, observe that $$F_{11}(z_1,z_2)=g_{ij}(z_1,z_2).$$
The following two lemmas can be proved along the same line of [\cite{MR3293430}, Proposition 4.1] and [\cite{MR3928691}, Lemma 3.2].
\begin{lemma}[\,\cite{MR3928691}\cite{MR3293430}]\label{H_0 for one variable}
    Let $\psi$ be any polynomial satisfying $\lvert\psi(z)\rvert\leq c\lvert z\rvert^s$ for some $c,s>0$. Let us define $$g(z)=f_1(z)+\frac{1}{1+z}f_2(0)$$ where $$f_j(z)=\frac{1}{2\pi i}\int_{\Gamma^i_j}\frac{\psi(\lambda)}{\lambda-z}d\lambda$$ for $j=1,2$.
    Then $g\in H^{\infty}_0(\Sigma_{\theta})$ and $\lVert g\rVert_{\infty,\Sigma_{\theta}}\lesssim\lVert \psi\rVert_{\infty,\Delta_{i}}$.
\end{lemma}
\begin{lemma}[\cite{MR3928691}, Lemma 3.2]
    Let us consider the function
    \begin{align*}
    h(z_1,z_2)=F_{11}(z_1,z_2)-\frac{1}{1+z_2}F_{12}(z_1,0)-\frac{1}{1+z_1}F_{21}(0,z_2) +\frac{F_{22}(0,0)}{(1+z_1)(1+z_2)}.
    \end{align*}
    Then $h\in H^{\infty}_0(\Sigma_{\theta}\times\Sigma_{\theta})$.\label{lemma for h}
\end{lemma}
\begin{lemma}\label{second h}
    There exists positive constant $K$ independent of $\phi$ such that 
    $$\lVert h \rVert_{\infty,\Sigma_{\theta}\times\Sigma_{\theta}}\leq K\lVert\phi\rVert_{\infty,\Delta \times \Delta}.$$
\end{lemma}
\begin{proof}
  For all  small $\epsilon>0$, consider $V_{l}=D(0,\epsilon)\cap\Delta^0_{l}$, $l=i,j.$ Note that, $d\left((\Sigma_{\theta}\setminus V_{l}),\Gamma^l_1\right)>0,l=i,j.$ Therefore, we have  $\lvert F_{11}(z_1,z_2)\rvert\lesssim\lVert\psi\rVert_{\infty,\Delta_{i}\times\Delta_{j}}$ for all $(z_1,z_2)\in(\Sigma_{\theta}\setminus V_i)\times(\Sigma_{\theta}\setminus V_{j})$. To this end rewrite $F_{11}$ as
\begin{align*}
  F_{11}(z_1,z_2)=\frac{1}{2\pi i}\int_{\Gamma^j_1}\frac{1}{\lambda_2-z_2}(\psi(z_1,\lambda_2)-\frac{1}{2 \pi i}\int_{\Gamma^i_2}\frac{\psi(\lambda_1,\lambda_2)}{(\lambda_1-z_1)}d\lambda_1)d\lambda_2.
 \end{align*}
 Note that, $d(z_1,\Gamma^i_2)>0$ and $d(z_2,\Gamma^j_1)>0$ for all $(z_1,z_2)\in V_i\times(\Sigma_{\theta}\setminus V_j)$. Hence , $\lvert F_{11}(z_1,z_2)\rvert\lesssim\lVert\psi\rVert_{\infty,\Delta_{i}\times\Delta_{j}}$ for all $(z_1,z_2)\in V_i\times(\Sigma_{\theta}\setminus V_j)$. Similarly, we will get the estimate on $(\Sigma_{\theta}\setminus V_i)\times V_j$.
 
 Let us define  $$F_{12}(z_1,z_2)=\frac{1}{2\pi i}\int_{\Gamma^j_2}\frac{1}{\lambda_2-z_2}\left(\psi(z_1,\lambda_2)-\frac{1}{2 \pi i}\int_{\Gamma^i_2}\frac{\psi(\lambda_1,\lambda_2)}{\lambda_1-z_1}d\lambda_1\right)d\lambda_2.$$
Since $d(\Gamma^j_2,V_i)>0$, we have the following
\begin{equation}
    \lvert F_{12}(z_1,z_2)\rvert\lesssim\lVert\psi\rVert_{\infty,\Delta_{i}\times\Delta_{j}} ,\qquad(z_1,z_2)\in V_i\times V_j.\label{v1-times-v1}
\end{equation}
\begin{equation}
\lvert F_{22}(z_1,z_2)\rvert\lesssim\lVert\psi\rvert_{\infty,\Delta_{i}\times\Delta_{j}},\qquad (z_1,z_2)\in V_i\times V_j.
\end{equation}
  Applying Cauchy's theorem, we obtain
  \begin{equation}
      \psi(z_1,z_2)=\sum_{1\leq l,m \leq 2}F_{lm}(z_1,z_2),\qquad(z_1,z_2)\in V_i\times V_j.\label{cauchy 2}
  \end{equation}
  Therefore, taking $F_{11}$ other side of equation \ref{cauchy 2}, we have
  \begin{equation}    
  \lVert F_{11}\rVert_{\infty,\Sigma_{\theta}\times\Sigma_{\theta}}\lesssim\lVert\psi\rVert_{\infty,\Delta_{i}\times\Delta_{j}}.\label{11}
\end{equation}

  Since $d(\Gamma^i_1,\Sigma_{\theta}\setminus V_i)>0$, we have that
  $\lvert F_{12}(z_1,0)\rvert\lesssim\lVert\psi\rVert_{\infty,\Delta_i\times\Delta_j}$ for all $z_1\in{\Sigma_{\theta}}\setminus V_i.$
  Therefore, from \ref{v1-times-v1}, we obtain \begin{equation}    
\lVert F_{12}(z_1,0)\rVert_{\infty,\Sigma_{\theta}}\lesssim \lVert\psi\rVert_{\infty,\Delta_i\times\Delta_j}.\label{12}
\end{equation}

Similarly, one can check that \begin{equation}
\lVert F_{21}(0,z_2)\rVert_{\infty,\Sigma_{\theta}}\lesssim \lVert\psi\rVert_{\infty,\Delta_i\times\Delta_j}.\label{21}
\end{equation}
Note that \begin{align*}
    F_{22}(0,0)=\frac{1}{(2\pi i)^2}\int_{\Gamma_2}\int_{\Gamma_2}\frac{\psi(\lambda_1,\lambda_2)}{\lambda_1\lambda_2}.
\end{align*}
Since $d(0,\Gamma^i_2)>0$ and $d(0,\Gamma^j_2)>0$, it follows that
\begin{equation}
\lvert F_{22}(0,0)\rvert\lesssim\lVert\psi\rVert_{\infty,\Delta_i\times\Delta_j}.\label{22}
\end{equation}

Combining the inequalities in \ref{11}, \ref{12}, \ref{21} and \ref{22}, we conclude that
\begin{equation}
    \lVert h\rVert_{\infty,\Sigma_{\theta}\times\Sigma_{\theta}}\leq K\lVert\psi\rVert_{\infty,\Delta_{i}\times\Delta_{j}},\qquad\text{for some} \hspace{0.2cm}K>0.
\end{equation} 
for some $K>0$. Since $\lVert \psi \rVert_{\infty,\Delta_{i}\times\Delta_{j}}=\lVert \phi\rVert_{\infty,\Delta\times\Delta}$ , it follows that
\begin{equation*}
    \lVert h \rVert_{\infty,\Sigma_{\theta}\times\Sigma_{\theta}}\leq K\lVert\phi\rVert_{\infty,\Delta \times \Delta}.
\end{equation*}
This completes the proof.
\end{proof}
 Let $A_1=A^1_p=I_{X}-\overline{\xi}_pT_1$ and $A_2=A^2_l=I_{X}-\overline{\xi}_qT_2$
. For any $ \rho\in(0,1)$ , define \begin{equation}
   \mathbf{A_{\rho}}:=({A^\rho_{1}}
,{A^\rho_{2}})=((1-\rho)I_{X}+\rho A_1,(1-\rho)I_{X}+\rho A_2).\label{eqn:}
    \end{equation}  Also, observe that $\sigma (A^\rho_{1}),\sigma ({A^\rho_{2}})\subset\Sigma_{\theta} .$

\begin{lemma}
Let $\xi_i$, $\xi_j\in \mathbb{T}$. 
Then $\lVert g_{ij}(\mathbf{A_\rho})\rVert\lesssim\lVert\phi\rVert_{\infty,\Delta\times\Delta}.$ \label{g_ij} 
\end{lemma}
\begin{proof}From Lemma \ref{lemma for h}, we have that \begin{align*}
    F_{11}(z_1,z_2)=h(z_1,z_2)-\frac{1}{1+z_2}F_{12}(z_1,0)-\frac{1}{1+z_1}F_{21}(0,z_2) +\frac{F_{22}(0,0)}{(1+z_1)(1+z_2)}.
    \end{align*}
    Therefore,
    \begin{align*}  
    g_{ij}(z_1,z_2)=h(z_1,z_2))-\frac{1}{(1+z_2)}F_{12}(z_1,0)\\\hspace{4 cm}&-\frac{1}{(1+z_1)}F_{21}(0,z_2))+\frac{F_{22}(0,0)}{(1+z_1))(1+z_2))}.
    \end{align*}
    As $h\in H^{\infty}_0(\Sigma_{\theta}\times\Sigma_{\theta})$ , by Lemma \ref{double h of A-rho}  $\lVert h(A_{\rho})\rVert\lesssim \lVert h\rVert_{\infty,\Sigma_{\theta}\times\Sigma_{\theta}}.$ Consider
    \begin{equation*}
    \tilde{h}(z_1))=F_{12}(z_1,0)+\frac{1}{(1+z_1)}F_{22}(0,0).
    \end{equation*}
     From Lemma \ref{p independent of h}, we have $\lVert \tilde{h}(A^{\rho}_{1})\rVert\lesssim\lVert\tilde{h}\rVert_{\infty,\Sigma_\theta}.$ It is immediate that
     \begin{equation}
     \lVert F_{12}(A^{\rho}_{1},0)\rVert\lesssim \lVert\phi\rVert_{\infty,\Delta\times\Delta}.\label{A1}
     \end{equation}
Similarly, we get the appropriate bound for $F_{21}$,
\begin{equation}
     \lVert F_{21}(0,{A^\rho_{2}})\rVert\lesssim \lVert\phi\rVert_{\infty,\Delta\times\Delta}.\label{A2}
     \end{equation}
     We have that 
     \begin{align*}
        g_{ij}(A_{\rho})=h(A_{\rho})+(I+A^{\rho}_2)^{-1}F_{12}(A^{\rho}_1,0)-(1+A^{\rho}_1)^{-1}F_{21}(0,A^{\rho}_2)\\+F_{22}(0,0)(1+A^{\rho}_1)^{-1}(1+A^{\rho}_2)^{-1}.
        \end{align*}
        From the Lemmas \ref{lemma for h} and \ref{second h} and the equations \ref{A1} and \ref{A2}, we have 
        $\lVert g_{ij}(A_{\rho})\rVert\lesssim\lVert\phi\rVert_{\infty,\Delta\times\Delta}.$
        This completes the proof of the lemma.
\end{proof}
 
 \begin{lemma}Let $\zeta_i\in\mathbb{T}$ and $\zeta_j\notin\mathbb{T}$.
Then $\lVert \phi_{ij}(\rho T_1,\rho T_2)\rVert\lesssim\lVert\phi\rVert_{\infty,\Delta\times\Delta}.$\label{zeta-i,zeta-j}
  \end{lemma}
   \begin{proof}For a fixed $\lambda_2$, consider the following function \begin{equation*}  
 \psi^{\lambda_2}(z_1)=\frac{1}{2\pi i}\int_{\gamma_i}\frac{\phi(\lambda_1,\lambda_2)}{(\lambda_1-z_1)}d\lambda_1\end{equation*}
We can write $\phi_{ij}$ in the following fashion
 \begin{align*}
 \phi_{ij}(z_1,z_2)=\frac{1}{2\pi i}\int_{\gamma_j}\frac{\psi^{\lambda_2}(z_1)}{(\lambda_2-z_2)}d\lambda_2.
 \end{align*}
 Let $g^{\lambda_2}(z_1)=\phi^{\lambda_2}(\xi_p(1-z_1))$. From[\cite{MR4819960}, Theorem 4.3], we have the following results
\begin{enumerate}
\item[(i)] $g^{\lambda_2}(A^{\rho}_1)=\psi^{\lambda_2}(\rho T_1)$ 
 \item[(ii)] $\lVert g^{\lambda_2}(A^{\rho}_1)\rVert\lesssim\lVert \phi^{\lambda_2} \rVert_{\infty,\Delta}$, where $\phi^{\lambda_2}(z_1)=\phi(z_1,\lambda_2)$ for a fixed $\lambda_2$.
 \end{enumerate}
 So, using the above results, we have that 
\begin{align*}
 \lVert\phi_{ij}(\rho T_1,\rho T_2)\rVert&=\left\|\frac{1}{2\pi i}\int_{\gamma_j}\phi^{\lambda_2}(\rho T_1)(\lambda_2-\rho T_2)^{-1}d\lambda_2\right\|\\&=\left\|\frac{1}{2\pi i}\int_{\gamma_j}g^{\lambda_2}(A^{\rho}_1)(\lambda_2-\rho T_2)^{-1}d\lambda_2\right\|\\&\lesssim \int_{\gamma_j }\|\phi^{\lambda_2}\|_{\infty,\Delta} \|(\lambda_2-\rho T_2)^{-1}\|d\lvert\lambda_2\rvert\\&\lesssim \lVert\phi\rVert_{\infty,\Delta\times\Delta}
\end{align*}This concludes the proof of the lemma.\end{proof}
\begin{lemma}
    Let $\zeta_i,\zeta_j\notin\mathbb{T}$. Then $\lVert \phi_{ij}(\rho T_1,\rho T_2)\rVert\lesssim\lVert\phi\rVert_{\infty,\Delta\times\Delta}.$\label{zeta_ij}
\end{lemma}
\begin{proof}
    Note that
    \begin{equation}
        \phi_{ij}(z_1,z_2)=\frac{1}{(2\pi i)^2}\int_{\gamma_i}\int_{\gamma_j}\frac{\phi(\lambda_1,\lambda_2)}{(\lambda_1-z_1)(\lambda_2-z_2)}d\lambda_1 d\lambda_2 .\label{phi 22}
    \end{equation}
Applying Dunford-Riesz functional calculus and Fubini's theorem, we have that
\begin{equation}
    \phi({\rho T_1,\rho T_2)}=\frac{1}{(2\pi i)^2}\int_{\gamma_i}\int_{\gamma_j}\phi(\lambda_1,\lambda_2)({\lambda_1}I-\rho T_1)^{-1}({\lambda_2I-\rho T_2)^{-1}} d\lambda_1 d\lambda_2
\end{equation}
 Since, $d(\gamma_i,\sigma(\rho T_1))>0$ and $d(\gamma_j,\sigma(\rho T_2))>0$ , it follows that $$\lVert \phi({\rho T_1,\rho T_2)}\rVert\lesssim \lVert \phi\rVert_{\infty,\Delta \times \Delta}.$$
 This concludes the proof of the lemma.
\end{proof}
\begin{proof}[Proof of Theorem \ref{transferprin}]
    Let $(A^1_i,A^2_j)$ has bounded $H^{\infty}(\Sigma_{\theta_{i}}\times\Sigma_{\theta_{j}})$- functional calculus. We may choose $\theta\in(0,\frac{\pi}{2})$ close enough to $\frac{\pi}{2}$ so that  $(A^1_{i},A^2_j)$ admits bounded $H^{\infty}(\Sigma_{\theta}\times \Sigma_{\theta})$ functional calculus for all
    $i,j\in\{1,2,\dots,N\}.$ In view of Proposition \ref{polygonally bounded}, the proof of this theorem reduces to proving that $(T_1,T_2)$ admits a joint polygonally bounded. Let $\phi\in\mathbb{C}[z_1,z_2]$ such that $\phi(0,0)=0$. From equation \ref{phi 22}, we have \begin{equation*} 
    \phi_{ij}(z_1,z_2)=(\frac{1}{2\pi i})^2 \int_{\gamma_i}\int_{\gamma_j}\frac{\phi(\lambda_1,\lambda_2)}{(\lambda_1-z_1)(\lambda_2-z_2)}d\lambda_1 d\lambda_2,
    \end{equation*}
   where  $\gamma_i,\gamma_j$ define as before for all $i,j\in\{1,2,\dots,N\}.$ Let the polygon $\Delta$ and $\Delta^0$ be defined as before with the vertices $\{\xi_1,c_1,\dots,\xi_{m},c_{m}\}$ and $\{\xi_1,d_1,\dots,\xi_{m},d_{m}\}$ respectively.
   
   Now, there will be three possibilities
\begin{enumerate}
    \item $\xi_i,\xi_j\in \mathbb{T}$
    \item $\xi_i\in\mathbb{T}$ but $\xi_j\notin\mathbb{T}$
    \item $\xi_i\notin\mathbb{T}$ and $\xi_j\notin\mathbb{T}.$
\end{enumerate}
For case(1), from equation \ref{g_{ij}}, we have that
\begin{align*}
    g_{ij}=\phi_{ij}(\xi_p(1-z_1),\xi_q(1-z_2)).
\end{align*}
Let $A_i=A^1_p=I-\overline{\xi}_pT_1$ and $A_j=A^2_{\rho}=I-\overline{\xi}_qT_2$. 
Let us define $A^{\rho}_{i}=(1-\rho)I+\rho A_i$ and $A^\rho_j=(1-\rho)I+\rho A_j$ for all $\rho\in(0,1).$
Also, one can observe that \begin{align*}
    g_{ij}(A^{\rho}_i,A^{\rho}_j)=\phi_{ij}(\rho T_1,\rho T_2).
\end{align*}

Applying Lemma \ref{g_ij}, we have
\begin{equation}
\lVert \phi_{ij}(\rho T_1,\rho T_2)\rVert\lesssim\lVert\phi_{ii}\rVert_{\infty,\Delta\times\Delta}.\label{phi-ij}
 \end{equation}
 For case (2), from Lemma \ref{zeta_ij} , we obtain that
 \begin{equation}
     \lVert\phi_{ij}(\rho T_1,\rho T_2)\rVert\lesssim\lVert\phi\rVert_{{\infty,\Delta\times\Delta}}.\label{phi1}
 \end{equation} 
 For case (3), by Lemma \ref{zeta-i,zeta-j} , it follows that
 \begin{equation}
     \lVert\phi_{ij}(\rho T_1,\rho T_2)\rVert\lesssim\lVert\phi\rVert_{\infty,\Delta\times\Delta}.\label{phi2}
 \end{equation}
  Applying Cauchy's theorem, it follows that \begin{equation}
     \lVert \phi(\rho T_1,\rho T_2)\rVert\lesssim\sum\limits_{1\leq i,j\leq m}\lVert\phi_{ij}(\rho T_1,\rho T_2)\rVert.\label{phi3}
 \end{equation}
 Combining the inequalities in \ref{phi1}, \ref{phi2} and \ref{phi3}, we obtain 
 \begin{equation*}
    \lVert \phi(\rho T_1,\rho T_2)\rVert\lesssim\lVert\phi\rVert_{\infty,\Delta\times\Delta}.
 \end{equation*}
 Taking $\rho\rightarrow 1$, we have that \begin{equation}
\lVert\phi(T_1,T_2)\rVert\lesssim\lVert\phi\rVert_{\infty,\Delta\times\Delta}.\label{phi}
\end{equation}
 Let $\phi\in\mathbb{C}[z_1,z_2]$ , and assume that $\tilde\phi(z_1,z_2)=\phi(z_1,z_2)-\phi(0,0)$. Therefore, $\tilde\phi(T_1,T_2)=\phi(T_1,T_2)-\phi(0,0)I.$
 By the inequality in \ref{phi}, we have that \begin{align*}
\lVert \tilde\phi(T_1,T_2)\rVert\lesssim\lVert\tilde\phi\rVert_{\infty,\Delta\times\Delta}.
\end{align*}
Therefore, $\lVert\phi(T_1,T_2)\rVert\lesssim\lVert\phi\rVert_{\infty,\Delta\times\Delta}$ for all $\phi\in\mathbb{C}[z_1,z_2].$

Conversely, suppose $(T_1,T_2)$ admits a bounded $H^{\infty}(E_{s_1}\times E_{s_2})$-functional calculus. By choosing $s\in(0,1)$, sufficiently close enough to $1,$ we ensure that $(T_1,T_2)$ admits bounded $H^{\infty}(E_s\times E_s)$- functional calculus. Observe that $E_s\subset\Sigma(\xi_j,\sin^{-1}s)$ for all $j=1,\dots,N$. For any $f\in H^{\infty}_{0}\left(\Sigma_{\sin^{-1}s}\times\Sigma_{\sin^{-1}s}\right)$ , we define \begin{equation}
    \phi(z_1,z_2)=f(1-\overline{\xi_i}z_1,1-\overline{\xi_j}z_2).
\end{equation}
Therefore, $\phi\in H^{\infty}_0(E_s\times E_s)$.
One can see that $\phi(T_1,T_2)=f(A_i,A_j)$.
Let us consider the sets $E^i_s=\{1-\overline{\xi}_i z: z \in E_s\}$ and $E^j_s=\{1-\overline{\xi_j}z:z\in E_s\}.$ By the assumption, we have that
\begin{align*}
    \lVert \phi(T_1,T_2)\rVert\lesssim\lVert\phi\rVert_{\infty,E_s\times E_s}.
\end{align*}
Therefore,\[\lVert f(A_1,A_2)\rVert\lesssim\lVert f\rVert_{\infty,E^i_s\times E^j_s} \lesssim\lVert f\rVert_{\infty,\Sigma_{\sin^{-1}s}\times\Sigma_{\sin^{-1}s}}.\]
This completes the proof of the theorem.
 \end{proof}

\section{Joint dilation for commuting tuple of $\text{Ritt}_E$ operators}\label{alldithm} This section is devoted to the proof of Theorem \ref{dilationthm} and Theorem \ref{CLASS}. We begin by describing some essential tools for this purpose. 

Consider the probability space $\Omega_0 = \{ \pm 1 \}^{\mathbb{Z}}$. For any $k \in \mathbb{Z}$, denote the $k$-th coordinate function $\varepsilon_k(\omega) = \omega_k$, where $\omega = \{\omega_j\}_{j \in \mathbb{Z}} \in \Omega_0$. The sequence of random variables $(\varepsilon_k)_{k\geq0}$ is called the Rademacher sequence on the probability space $(\Omega_0, \mathbb{P})$. For $1 \leq p < \infty$, we denote $\text{Rad}_p(X)$ to be the span of $\{\varepsilon_k \otimes x_k \mid x_k \in X, k \in \mathbb{Z}\}$ in the Bochner space $L^p(\Omega_0, X)$. For $p=2$, we simply denote $\text{Rad}(X)$.

Let $M\subseteq B(X)$. We say that $M$ is $R$-bounded if there exists $C>0,$ such that for any finite sequence
$(T_k)_{k=0}^N$ in $M$ and a finite sequence $(x_k)_{k=0}^N$ in $X,$ we have
\begin{equation*}\label{R}
\Big\|\sum_{k=0}^N\varepsilon_k\otimes T_k(x_k)\Big\|_{\text{Rad(X)}}\leq C\Big\|\sum_{k=0}^N\varepsilon_k\otimes x_k\Big\|_{\text{Rad(X)}}.
\end{equation*}
An operator \( T : X \to X \) on a Banach space \( X \) is said to be \( R\text{-Ritt}_E \) if there exists \( r \in (0,1) \) such that
\[
\sigma(T) \subset \overline{E_r}
\]
and for all \( s \in (r,1) \), the set
\[
\left\{ 
R(z, T) \prod_{j=1}^N (\xi_j - z) : z \in D(0,2) \setminus \overline{E_s} 
\right\} 
\]
is \( R \)-bounded. We refer  \cite{bouabdillah2024squarefunctionsassociatedritte} for more on \( R\text{-Ritt}_E\) opertaors. The notion of $R$-sectorial operator is similarly defined (see \cite{MR3928691}).
\begin{definition}(Square Function)
Let $T: X \to X$ be a $\RE$ operator and let $\alpha>0$ .  For any $x \in X$ the square function $\|x\|_{T,\alpha}$ associated with $T$ is defined by
\[
\|x\|_{T,\alpha} = \lim_{n \to \infty} \Big\|  \sum_{k=1}^{n} k^{\alpha-\frac{1}{2}} \varepsilon_k \otimes T^{(k-1)}\prod^N_{j=1}(I -\overline{ \xi_j} {T})^\alpha x \Big\|
_{\text{Rad}(X)}.
\]
\end{definition}
If $T$ is a $\RE$ operator, then $(I - \overline{\xi_i}T)$ is a sectorial operator. Therefore, the fractional power of $(I - \overline{\xi_i}T)$ is well-defined.
Therefore,
$ \|x\|_{T,\alpha} $ is well-defined for all $x \in X$.

\begin{theorem}(Square function estimate)\label{Square function estimate}[\cite{bouabdillah2024squarefunctionsassociatedritte}, Proposition 6.3 ]
Assume that $X$ is a Banach space with finite cotype. Let $T : X \to X$ be a Ritt$_{E}$ operator of type $r \in (0,1)$ and assume that $T$ admits a bounded $H^\infty(E_s)$-functional calculus for some $s \in (r,1)$. Then for any $\alpha > 0$, we have an estimate
\[
\|x\|_{T,\alpha} \lesssim \|x\|, \quad \text{for all}\qquad x \in X.
\]
\end{theorem}
\begin{theorem}(Kahane contraction Principle)[\cite{2}, Theorem 3.1]\label{kahane contraction principle}
    \textit{Let \( (X_n)_{n=1}^{\infty} \) be a sequence of independent symmetric  Banach space valued random variables. Then for all \( a_1, \dots, a_N \in \mathbb{R} \) and \( 1 \leq p < \infty \),}
\begin{equation}
\mathbb{E} \left\| \sum_{n=1}^{N} a_n X_n \right\|^p 
\leq \left( \max_{1 \leq n \leq N} |a_n| \right)^p 
\mathbb{E} \left\| \sum_{n=1}^{N} X_n \right\|^p.
\end{equation}
\end{theorem}
Let $(a_m)_m$ be a sequence of complex numbers provided by 
\begin{equation}
    \frac{1}{\prod\limits^N_{i=1}(1-\overline\xi_jz)}=\sum^{\infty}_{m=0}a_mz^m.
\end{equation}
From [\cite{MR4861029}, Lemma  3.2] we have that the sequence $(a_m)_m$ is bounded.
\begin{theorem}[\cite{MR4819960}, Proposition 3.5]
Let $T\in B(X)$ and $T$ has bounded polygonal functional calculus iff there exist finite subset $E$ of $\mathbb{T}$ and $s\in(0,1)$ such that $T$ is $\RE$ and $T$ has bounded $H^{\infty}(E_s)$- functional calculus. \label{poly to ritt_E} 
\end{theorem}
The proof of the following theorem is similar to that in [\cite{MR4861029}, Lemma 3.3], so we omit it.
\begin{lemma}
    Let $X$ be a Banach space with finite cotype and $T\in B(X)$, has polygonally bounded then for any $ x\in \overline{Ran}\left(\prod\limits^N_{i=1}(I_{X}-\overline{\xi_i}T)\right)$,
    we have $$\sum^{\infty}_{m=0}a_mT^m\prod^N_{j=1}(I_{X}-\overline{\xi_j}T)x=x.$$\label{fixedx lemma}
\end{lemma}

\begin{theorem}[\cite{MR77411}, Theorem 1.3]Let $T\in B(X)$, where $X$ is a reflexive Banach space. 
    If $T$ is power bounded, then the following conditions hold
    \begin{enumerate}
        \item The space $X$ admits the decomposition
             $$X=\text{\text{Ker}}(I_{X}-T)\oplus\overline{Ran}(I_{X}-T)$$
    \item For any $z\in\overline{Ran}(I_{X}-T)$ and any $y\in \text{\text{Ker}}(I_{X}-T)$ we have $$\langle z,y\rangle=0.$$
   \end{enumerate}\label{Mean ergodic theorem 4}
\end{theorem}
In a manner analogous to proof of [\cite{MR4861029}, Lemma 3.4], we can establish the following lemma.
\begin{lemma}
   Let $T$ be $\RE$ operator. Suppose $X$ is reflexive Banach space then the following decompositions hold
    \begin{equation}
        X=\bigoplus^N_{j=1}\text{\text{Ker}}(I_{X}-\overline{\xi_j}T)\bigoplus\overline{Ran}\left(\prod\limits^N_{j=1}(I_{X}-\overline{\xi_j}T)\right)\label{Mean ergodic theorem 1}
    \end{equation}
    \begin{equation}
        X^*=\bigoplus^N_{j=1}\text{\text{Ker}}(I_{X}-{\xi_j}T^*)\bigoplus\overline{Ran}\left(\prod^N_{j=1}(I_{X}-{\xi_j}T^*)\right)\label{Mean ergodic theorem 2}
    \end{equation}
    \begin{equation}
        \overline{Ran}(I_X-\overline{\xi_1}T)=\bigoplus^{N}_{j=2}\text{\text{Ker}}(I_X-\overline{\xi_j}T)\bigoplus\overline{Ran}\left(\prod^N_{j=1}(I_X-\overline{\xi_j}T)\right)
    \end{equation}
\end{lemma}\label{Mean ergodic theorem 3}
Further, we can see that $N+1$ projections associated with the decomposition of $X$ belong to the set of bicommutants of $T$.
\begin{remark}
    From the above lemma, it follows that any $x\in X$ admits a unique decomposition $x=x_1+x_2+\dots+x_{N+1}$ where $x_i\in \text{Ker}(I_X-\overline{\xi_i}T)$ for $i=1,2,\dots,N$ and $x_{N+1}\in\overline{Ran}\left(\prod\limits^N_{i=1}(I_X-\overline{\xi_i}T)\right)$ . Any $y\in X^*$, $y$ can be uniquely written as 
    $y=y_1+y_2+\dots+y_{N+1}$ where $y_i\in \text{Ker}(I_{X}-\xi_jT^*)$ for all $i=1,2,\dots,N$ and $y_{N+1}\in\overline{Ran}\left(\prod\limits^N_{i=1}(I_X-\xi_iT^*)\right)$.
    By applying Theorem \ref{Mean ergodic theorem 4}, one can easily check that $\langle x_j,y_i\rangle=0$ where $1\leq i\neq j\leq N+1$.\label{orthogonality}
\end{remark}
\begin{theorem}\label{dilation}
     Let $X$ be a reflexive Banach space such that both $X$ and $X^*$ have finite cotype. Let $T\in B(X)$ be $\RE$ operator with bounded $H^{\infty}(E_s)$- functional calculus. Then, for any $p\in(1,\infty)$, there exist two bounded operators $J:X\rightarrow L^p(\Omega,X)$ , $Q:L^p(\Omega,X)\rightarrow X$ and an isometry $V:L^p(\Omega,X)\rightarrow L^p(\Omega,X)$ such that $T^n=QV^nJ$ for all $n\geq0$ .
\end{theorem} 
\begin{proof}    Let $L^p(\Omega,X)=L^p([N],X)\oplus L^p(\Omega_0,X)$.
    Consider the map $D:L^p([N],X)\rightarrow L^p([N],X)$ defined by $D(x_1,x_2,..,x_N)=(\xi_1x_1,\xi_2x_2,\dots,\xi_N x_N)$ where $\xi_j\in E$ for all $j=1,\dots,N$. Let us define an operator $u: L^p(\Omega_0) \to L^p(\Omega_0)$ as $u(f)(\{\omega_k\}_k) = f(\{\omega_{k-1}\}_k)$ for $f \in L^p(\Omega_0)$. Clearly, $u$ is a positive isometric isomorphism. Therefore, one can extend $u$ to an operator $U:= u \otimes I_X$ on the Banach space $L^p(\Omega_0, X)$ as an isometric isomorphism. Since $u(\epsilon_k) = \epsilon_{k-1}$, for any $\sum_k \epsilon_k \otimes x_k \in \text{Rad}_p(X)$, we have $U(\sum_k \epsilon_k \otimes x_k) = \sum_k \epsilon_k \otimes x_{k+1}$.

    Let $V=D\oplus U^2$ . From  equation \ref{Mean ergodic theorem 1}, we have that $x=x_1+x_2+\dots+x_N+x_{N+1}$ , $x_i\in \text{Ker}(I_{X}-\overline\xi_iT)$ where $i=1,2,...,N$ and $x_{N+1}\in\overline{Ran}\left(\prod\limits^N_{j=1}(I_{X}-{\xi_j}T)\right)$. Let $ A=\prod\limits^N_{j=1}(I_{X}-\xi_jT)^{\frac{1}{2}}$ . Consider the map $J:X\rightarrow L^p(\Omega,X)$ defined by 
    \begin{equation}
    J(x)=(x_1,\dots,x_N)\oplus\left(\sum^{\infty}_{k=0}\varepsilon_{2k}\otimes T^kAx_{N+1}+\sum^{\infty}_{k=0}\varepsilon_{2k+1}\otimes T^{k+1}Ax_{N+1}\right).\label{the map J}
    \end{equation}
    By the square function estimate of the $\RE$ operator, one can see that the map $J$ is well-defined.
    From  equation \ref{Mean ergodic theorem 2}, $y=y_1+y_2+\dots+y_N+y_{N+1}$ for all $y\in X^*$.
     Consider the map $\Tilde{J}:X^*\rightarrow L^p(\Omega,X)^*$ defined by 
    \begin{align*}
         \tilde{J}(y)=(y_1,y_2,\dots,y_N)\oplus& & & & & & & & 
         \end{align*}$$\hspace{2.5cm}\left(\sum^{\infty}_{k=0}\varepsilon_k\otimes \overline{a_{2k}} {T^*}^kA^*y_{N+1}+\sum^{\infty}_{k=0}\varepsilon_{2k-1}\otimes \overline{a_{2k-1}}  {T^*}^{(k-1)}A^*y_{N+1}\right).$$ 
         
By Theorem \ref{kahane contraction principle} and square function estimate for $T^*$, the map $\tilde{J}$ is well-defined.
Therefore,

    \begin{align*}
        V^nJ(x)=(\xi^n_1x_1,\dots,\xi^n_Nx_N)\oplus & & & & & & & & & 
    \end{align*} $$ \hspace{3cm}\left(\sum_{k=0}^{\infty}\varepsilon_{2k}\otimes T^{n+k}Ax_{N+1} + \sum^{\infty}_{k=0}\varepsilon_{2k+1}\otimes T^{n+k+1}Ax_{N+1}\right).
$$   
 Here $\varepsilon_{2k}\otimes T^{n+k}Ax_{N+1}$  appears in $2k$-th position and $\varepsilon_{2k+1}\otimes T^{n+k+1}Ax_{N+1}$ appears in $(2k+1)$-th position.
Let $W_1=\sum^{\infty}_{k=0}a_{2k}\langle T^{n+k}Ax_{N+1},{T^*}^kA^*y_{N+1}\rangle$ and $W_2=\sum^{\infty}_{k=0}a_{2k+1}\langle T^{n+k+1}Ax_{N+1},{T^*}^kA^*y_{N+1}\rangle$.
From the Remark \ref{orthogonality}, we have 
\begin{align*}
    \langle V^nJ(x),\tilde{J}(y)\rangle&=\sum^N_{j=1}\xi^n_j\langle x_j,y_j\rangle+W_1+W_2
    \\&=\sum^N_{j=1}\xi^n_j\langle x_j,y_j\rangle+\sum^{\infty}_{m=0}\langle a_mT^{n+m}A^2x_{N+1},y_{N+1}\rangle.
\end{align*}
From the Lemma \ref{fixedx lemma}, we have 
\begin{equation}
    \langle V^nJ(x),\tilde{J}(y)\rangle=\sum^N_{j=1}\xi^n_j\langle x_j,y_j\rangle+\langle T^nx_{N+1},y_{N+1}\rangle.
\end{equation}
As $T^nx_j=\xi^n_jx_j$ where $j=1,\dots,N.$ So, the above equation becomes
\begin{align}
    \langle V^nJ(x),\tilde{J}(y)\rangle=\sum^N_{j=1}\langle T^nx_j,y_j\rangle+\langle T^nx_{N+1},y_{N+1}\rangle.\label{Vn}
\end{align}
The decomposition of $T^n(x)$ with respect to equation \ref{Mean ergodic theorem 1} is 
$T^n(x)=T^n(x_1)+\dots+T^n(x_{N+1})$. 
So, the  equation \ref{Vn} becomes
$\langle QV^nJ(x),y \rangle=\langle T^n(x),y\rangle$ where $Q={\tilde{J}}^*$ and $y\in X^*$.
Therefore, $QV^nJ=T^n$ for all $n\geq0$.
\end{proof}
\begin{lemma}
    Let $X$ be a reflexive Banach space such that both $X$ and $X^*$ have finite cotype. Let $T$ be  $\RE$ which has bounded $H^{\infty}(E_s)$- functional calculus. Let $S\in B(X)$ such that $ST=TS$. Then $(\oplus^N S\oplus(I_{L^p(\Omega_{0})}\otimes S)J=JS$ where the map $J$  define as \ref{the map J}.\label{composion of J and S}
    \end{lemma}
\begin{proof}
For any  $x\in X$, observe that
  \begin{align}
      (\oplus^NS\oplus(I_{L^p(\Omega_{0})}\otimes S))J(x)& & & & & & & &\label{the com of S and J}\end{align}$$=(Sx_1,\dots,Sx_N)\oplus\left(\sum^{\infty}_{k=0}\varepsilon_{2k}\otimes T^kASx_{N+1}+\sum^{\infty}_{k=0}\varepsilon_{2k+1}\otimes T^{k+1}ASx_{N+1}\right)$$
Since, $x_i\in \text{Ker}(I_X-\overline{\xi_i}T)=0$ and $ST=TS$ , it follows that $(I_X-\overline{\xi}T)Sx_i=0$; that is $Sx_i\in \text{Ker}(I_X-\overline\xi_iT)$ for all $i=1,\dots,N$.
For any $x\in X$, we have 
     $x=x_1+\dots+x_{N+1}$ . Then $Sx$ admits the unique decomposition
     $Sx=Sx_1+\dots+Sx_{N+1}.$ Therefore, $Sx_i\in \text{Ker}(I-\overline{\xi_i}T)$ and
      $Sx_{N+1}\in\overline{Ran}\left(\prod\limits^N_{j=1}(I_X-\overline{\xi_j}T)\right)$. By the definition of the map $J$, it is immediate that \begin{equation}   
      JS(x)=(Sx_1,\dots,Sx_N)\oplus\left(\sum\limits^{\infty}_{k=0}\varepsilon_{2k}\otimes T^kASx_{N+1}+\sum\limits^{\infty}_{k=0}\varepsilon_{2k+1}\otimes T^{k+1}ASx_{N+1}\right).\label{JS}
      \end{equation}
      Combining the equations \ref{the com of S and J} and \ref{JS}, the lemma is proven.
\end{proof}

\begin{proof}[Proof of Theorem \ref{dilationthm}]
Assume that $L^p(\Omega',X)=L^p([N],L^p(\Omega,X))\bigoplus L^p(\Omega_0,L^p(\Omega,X))$.
 Let $D_{L^p(\Omega,X)}:L^p([N],L^p(\Omega,X))\rightarrow L^p([N],L^p(\Omega,X)$ such that \begin{equation}
     D_{L^p(\Omega,X)}(x_1,\dots,x_N)=(\xi_1x_1,\dots,\xi_Nx_N),
 \end{equation} where $x_i\in L^p(\Omega,X)$ for all $i\in\{1,2,\dots,N\}$.
  From  Theorem \ref{dilation} , we have that  $T^{i_1}_1=Q_1V^{i_1}_1J_1$ and $T^{i_2}_2=Q_2V^{i_2}_2J_2$. Therefore,
     $$ T^{i_1}_1T^{i_2}_2=Q_1V^{i_1}_1J_1T^{i_2}_2 $$
  Applying  Lemma \ref{composion of J and S}, we have 
  \begin{align*}
  T^{i_1}_1T^{i_2}_2&=Q_1V^{i_1}_1(\oplus^N T^{i_2}_2\oplus( I_{L^p(\Omega_0)}\otimes T^{i_2}_2))J_1\\&=Q_1V^{i_1}_1\left(\oplus^N Q_2\oplus(I_{L^p(\Omega_0)}\otimes Q_2)\right)\left(\oplus^N V^{i_2}_2\oplus(I_{L^p(\Omega_0)}\otimes V^{i_2}_2\right)\\ &   \qquad \qquad\qquad\qquad\qquad(\oplus^N J_2\oplus(I_{L^p(\Omega_0)}\otimes J_2))J_1\\
  &=Q_1(\oplus^N Q_2\oplus(I_{L^p(\Omega_0)}\otimes Q_2))(D_{L^p(\Omega,X)}\oplus(u\otimes I_{L^p(\Omega,X)})^2)^{i_1}\\&\qquad\qquad\left(\oplus^N V^{i_2}_2\oplus(I_{L^p(\Omega_0)}\otimes V^{i_2}_2)\right)(\oplus^N J_2\oplus(I_{L^p(\Omega_0)}\otimes J_2))J_1\\&=Q(D_{L^p(\Omega,X)}\oplus(u\otimes I_{L^p(\Omega,X)})^{2})^{i_1}\left(\oplus^N V^{i_2}_2\oplus(I_{L^p(\Omega_0)}\otimes V^{i_2}_2)\right)J
  \\&=Q(D_{L^p(\Omega,X)}\oplus(u\otimes I_{L^p(\Omega,X)})^{2})^{i_1}\left(\oplus^N V_2\oplus(I_{L^p(\Omega_0)}\otimes V_2)\right)^{i_2}J
  \end{align*} 
  Where $Q=Q_1(\oplus^N Q_2\oplus(I_{L^p(\Omega_0)}\otimes Q_2))$ and $J=(\oplus^N J_2\oplus(I_{L^p(\Omega_0)}\otimes J_2))J_1$.
  It is easy to check that the operators $(D_{L^p(\Omega,X)}\oplus(u\otimes I_{L^p(\Omega,X)})^{2})$ and $\left(\oplus^N V_2\oplus(I_{L^p(\Omega_0)}\otimes V_2)\right)$ commute.
\end{proof}
\begin{proof}[Proof of Theorem \ref{CLASS}]
 Note that $(2)$ follows from $(1)$ by applying \cite{bouabdillah2024squarefunctionsassociatedritte}[Proposition 6.6] and Theorem \ref{dilationthm}. Now suppose $(2)$ holds,$(3)$ follows from the matrix valued version of Coifmann-Weiss general transference principle[\cite{MR481928},Theorem 2.4]. From $(3)$ to $(4)$ follow trivially. Let us assume $(4)$. Then each $\overline{\xi_j}T_i, j=1,2\dots,N$ and $i=1,2$ is $p$-polynomially bounded, from [\cite{MR3265289}, Proposition 4.7] and we can obtain $(5)$. Using Theorem \ref{transferprin} and the joint functional calculus property of $L^p$-spaces we get $(1)$ from $(5)$. This completes the proof of the theorem. 
\end{proof}
We now consider $\text{Ritt}_E$ operators on Hilbert space to prove an analogue of Theorem \ref{CLASS}. We refer to \cite{MR3928691} for the proof of the theorem for the Ritt case. The proof of the $\text{Ritt}_E$ follows in a similar way using the transfer principle, Theorem \ref{transferprin}, dilation Theorem \ref{dilationthm} and joint functional calculus property of Hilbert spaces \cite{MR1635157}.
\begin{theorem}\label{JRS}
Let $E=\{\xi_1,\dots,\xi_N\}$ be a finite subset of $\mathbb{T}$. Let $(T_1,T_2)$ be a commuting tuple of $\RE$ operators on a Hilbert space $\mathcal{H}.$ Then, the following assertions are equivalent.
\begin{itemize}
\item[1.] Each $T_i$ is similar to a contraction, $1\leq i\leq 2.$
\item[2.] The tuple $(T_1,T_2)$  admits a joint bounded $H^\infty$-functional calculus.
\item[3.]The tuple $(T_1,T_2)$ is jointly similar to a commuting tuple of contractions.
\end{itemize}
\end{theorem}
\textbf{Data Availability}

This study does not use or generate any datasets.
\section*{Acknowledgements}
The third named author thanks the DST-INSPIRE Faculty Fellowship \\
DST/INSPIRE/04/2020/001132 and Prime Minister Early Career Research Grant Scheme ANRF/ECRG/2024/000699/PMS and ANRF/ARGM/2025/000895/MTR. 

\bibliographystyle{amsplain}
\nocite{*}
\bibliography{main}

@article {MR4819960,
    AUTHOR = {Bouabdillah, Oualid and Le Merdy, Christian},
     TITLE = {Polygonal functional calculus for operators with finite
              peripheral spectrum},
   JOURNAL = {Israel J. Math.},
  FJOURNAL = {Israel Journal of Mathematics},
    VOLUME = {263},
      YEAR = {2024},
    NUMBER = {2},
     PAGES = {517--551},
      
}

@article {MR4861029,
    AUTHOR = {Le Merdy, Christian and reshmi, M. N.},
     TITLE = {Commuting families of polygonal type operators on {H}ilbert
              space},
   JOURNAL = {Adv. Oper. Theory},
  FJOURNAL = {Advances in Operator Theory},
    VOLUME = {10},
      YEAR = {2025},
    NUMBER = {2},
     PAGES = {Paper No. 33},
      
}

@book {MR77411,
    AUTHOR = {Krengel, Ulrich},
     TITLE = {Ergodic theorems},
    SERIES = {De Gruyter Studies in Mathematics},
    VOLUME = {6},
      NOTE = {With a supplement by Antoine Brunel},
 PUBLISHER = {Walter de Gruyter \& Co., Berlin},
      YEAR = {1985},
     PAGES = {viii+357},
      
}

@article {MR4043874,
    AUTHOR = {Arrigoni, Olivier and Le Merdy, Christian},
     TITLE = {{$H^\infty$}-functional calculus for commuting families of
              {R}itt operators and sectorial operators},
   JOURNAL = {Oper. Matrices},
  FJOURNAL = {Operators and Matrices},
    VOLUME = {13},
      YEAR = {2019},
    NUMBER = {4},
     PAGES = {1055--1090},
     
}

@incollection {MR1768324,
    AUTHOR = {Le Merdy, Christian},
     TITLE = {{$H^\infty$}-functional calculus and applications to maximal
              regularity},
 BOOKTITLE = {Semi-groupes d'op\'erateurs et calcul fonctionnel ({B}esan\c
              con, 1998)},
    SERIES = {Publ. Math. UFR Sci. Tech. Besan\c con},
    VOLUME = {16},
     PAGES = {41--77},
 PUBLISHER = {Univ. Franche-Comt\'e, Besan\c con},
      YEAR = {1999},
   
}

@article {bouabdillah2024squarefunctionsassociatedritte,
    AUTHOR = {Bouabdillah, Oualid},
     TITLE = {Square functions associated with {${\rm Ritt}_E$} operators},
   JOURNAL = {Indag. Math. (N.S.)},
  FJOURNAL = {Koninklijke Nederlandse Akademie van Wetenschappen.
              Indagationes Mathematicae. New Series},
    VOLUME = {36},
      YEAR = {2025},
    NUMBER = {5},
     PAGES = {1417--1452},
      ISSN = {0019-3577,1872-6100},
   MRCLASS = {47A60 (47B01 47B12)},
  MRNUMBER = {4949890},
       DOI = {10.1016/j.indag.2025.05.014},
       URL = {https://doi.org/10.1016/j.indag.2025.05.014},
}

@book{2,
    AUTHOR = {Jan van Neerven},
     TITLE = {Stochastic Evolution Equations},
    SERIES = {},
  PUBLISHER={ISEM Lecture Notes },
       YEAR={2007},
    URL={J.M.A.M.vanNeerven-at-tudelft.nl },
}

@article {MR3293430,
    AUTHOR = {Le Merdy, Christian},
     TITLE = {{$H^\infty$} functional calculus and square function estimates
              for {R}itt operators},
   JOURNAL = {Rev. Mat. Iberoam.},
  FJOURNAL = {Revista Matem\'atica Iberoamericana},
    VOLUME = {30},
      YEAR = {2014},
    NUMBER = {4},
     PAGES = {1149--1190},
      
}

@article {MR43386,
    AUTHOR = {von Neumann, Johann},
     TITLE = {Eine {S}pektraltheorie f\"ur allgemeine {O}peratoren eines
              unit\"aren {R}aumes},
   JOURNAL = {Math. Nachr.},
  FJOURNAL = {Mathematische Nachrichten},
    VOLUME = {4},
      YEAR = {1951},
     PAGES = {258--281},
      
}

@article {MR155193,
    AUTHOR = {And\^o, T.},
     TITLE = {On a pair of commutative contractions},
   JOURNAL = {Acta Sci. Math. (Szeged)},
  FJOURNAL = {Acta Universitatis Szegediensis. Acta Scientiarum
              Mathematicarum},
    VOLUME = {24},
      YEAR = {1963},
     PAGES = {88--90},
      
}

@article {MR355642,
    AUTHOR = {Varopoulos, N. Th.},
     TITLE = {On an inequality of von {N}eumann and an application of the
              metric theory of tensor products to operators theory},
   JOURNAL = {J. Functional Analysis},
  FJOURNAL = {Journal of Functional Analysis},
    VOLUME = {16},
      YEAR = {1974},
     PAGES = {83--100},
     
}

@article {MR4611833,
    AUTHOR = {Hong, Guixiang and Ray, Samya Kumar and Wang, Simeng},
     TITLE = {Maximal ergodic inequalities for some positive operators on
              noncommutative {$L_p$}-spaces},
   JOURNAL = {J. Lond. Math. Soc. (2)},
  FJOURNAL = {Journal of the London Mathematical Society. Second Series},
    VOLUME = {108},
      YEAR = {2023},
    NUMBER = {1},
     PAGES = {362--408},
      
}

@article {MR4092689,
    AUTHOR = {Ray, Samya Kumar},
     TITLE = {On multivariate {M}atsaev's conjecture},
   JOURNAL = {Complex Anal. Oper. Theory},
  FJOURNAL = {Complex Analysis and Operator Theory},
    VOLUME = {14},
      YEAR = {2020},
    NUMBER = {4},
     PAGES = {Paper No. 42, 25},
      
}

@article {MR3928691,
    AUTHOR = {Mohanty, Parasar and Ray, Samya Kumar},
     TITLE = {On joint functional calculus for {R}itt operators},
   JOURNAL = {Integral Equations Operator Theory},
  FJOURNAL = {Integral Equations and Operator Theory},
    VOLUME = {91},
      YEAR = {2019},
    NUMBER = {2},
     PAGES = {Paper No. 14, 18},
      
}

@article {MR3897976,
    AUTHOR = {Gupta, Rajeev and Ray, Samya K.},
     TITLE = {On a question of {N}. {T}h. {V}aropoulos and the constant
              {$C_2(n)$}},
   JOURNAL = {Ann. Inst. Fourier (Grenoble)},
  FJOURNAL = {Universit\'e{} de Grenoble. Annales de l'Institut Fourier},
    VOLUME = {68},
      YEAR = {2018},
    NUMBER = {6},
     PAGES = {2613--2634},
     
}

@incollection {MR912940,
    AUTHOR = {McIntosh, Alan},
     TITLE = {Operators which have an {$H_\infty$} functional calculus},
 BOOKTITLE = {Miniconference on operator theory and partial differential
              equations ({N}orth {R}yde, 1986)},
    SERIES = {Proc. Centre Math. Anal. Austral. Nat. Univ.},
    VOLUME = {14},
     PAGES = {210--231},
 PUBLISHER = {Austral. Nat. Univ., Canberra},
      YEAR = {1986},
     
}

@article {MR1364554,
    AUTHOR = {Cowling, Michael and Doust, Ian and McIntosh, Alan and Yagi,
              Atsushi},
     TITLE = {Banach space operators with a bounded {$H^\infty$} functional
              calculus},
   JOURNAL = {J. Austral. Math. Soc. Ser. A},
  FJOURNAL = {Australian Mathematical Society. Journal. Series A. Pure
              Mathematics and Statistics},
    VOLUME = {60},
      YEAR = {1996},
    NUMBER = {1},
     PAGES = {51--89},
     
}

@book {MR1818047,
    AUTHOR = {Pisier, Gilles},
     TITLE = {Similarity problems and completely bounded maps},
    SERIES = {Lecture Notes in Mathematics},
    VOLUME = {1618},
   EDITION = {expanded},
      NOTE = {Includes the solution to ``The Halmos problem''},
 PUBLISHER = {Springer-Verlag, Berlin},
      YEAR = {2001},
     PAGES = {viii+198},
      
}

@book {MR1976867,
    AUTHOR = {Paulsen, Vern},
     TITLE = {Completely bounded maps and operator algebras},
    SERIES = {Cambridge Studies in Advanced Mathematics},
    VOLUME = {78},
 PUBLISHER = {Cambridge University Press, Cambridge},
      YEAR = {2002},
     PAGES = {xii+300},
      
}

@article{hartz2025neumann,
  title={On von Neumann’s inequality on the polydisc},
  author={Hartz, Michael},
  journal={Mathematische Annalen},
  volume={391},
  number={4},
  pages={5235--5264},
  year={2025},
  publisher={Springer}
}

@book {MR275190,
    AUTHOR = {Sz.-Nagy, B\'ela and Foia\c s, Ciprian},
     TITLE = {Harmonic analysis of operators on {H}ilbert space},
      NOTE = {Translated from the French and revised},
 PUBLISHER = {North-Holland Publishing Co., Amsterdam-London; American
              Elsevier Publishing Co., Inc., New York; Akad\'emiai Kiad\'o,
              Budapest},
      YEAR = {1970},
     PAGES = {xiii+389},
  
}

@book {MR2244037,
    AUTHOR = {Haase, Markus},
     TITLE = {The functional calculus for sectorial operators},
    SERIES = {Operator Theory: Advances and Applications},
    VOLUME = {169},
 PUBLISHER = {Birkh\"auser Verlag, Basel},
      YEAR = {2006},
     PAGES = {xiv+392},
      
}

@phdthesis{albrecht1994functional,
  title={Functional calculi of commuting unbounded operators},
  author={Albrecht, David William},
  year={1994},
  school={Monash University}
}

@article {MR1635157,
    AUTHOR = {Lancien, Florence and Lancien, Gilles and Le Merdy, Christian},
     TITLE = {A joint functional calculus for sectorial operators with
              commuting resolvents},
   JOURNAL = {Proc. London Math. Soc. (3)},
  FJOURNAL = {Proceedings of the London Mathematical Society. Third Series},
    VOLUME = {77},
      YEAR = {1998},
    NUMBER = {2},
     PAGES = {387--414},
     
}

@article {MR2980915,
    AUTHOR = {Le Merdy, Christian and Xu, Quanhua},
     TITLE = {Maximal theorems and square functions for analytic operators
              on {$L^p$}-spaces},
   JOURNAL = {J. Lond. Math. Soc. (2)},
  FJOURNAL = {Journal of the London Mathematical Society. Second Series},
    VOLUME = {86},
      YEAR = {2012},
    NUMBER = {2},
     PAGES = {343--365},
     
}

@article {MR3060752,
    AUTHOR = {Le Merdy, Christian and Xu, Quanhua},
     TITLE = {Strong {$q$}-variation inequalities for analytic semigroups},
   JOURNAL = {Ann. Inst. Fourier (Grenoble)},
  FJOURNAL = {Universit\'e{} de Grenoble. Annales de l'Institut Fourier},
    VOLUME = {62},
      YEAR = {2012},
    NUMBER = {6},
     PAGES = {2069--2097},
     
}

@article {MR1837537,
    AUTHOR = {Blunck, S\"onke},
     TITLE = {Analyticity and discrete maximal regularity on {$L_p$}-spaces},
   JOURNAL = {J. Funct. Anal.},
  FJOURNAL = {Journal of Functional Analysis},
    VOLUME = {183},
      YEAR = {2001},
    NUMBER = {1},
     PAGES = {211--230},
     
}

@article {MR3265289,
    AUTHOR = {Arhancet, C\'edric and Le Merdy, Christian},
     TITLE = {Dilation of {R}itt operators on {$L^p$}-spaces},
   JOURNAL = {Israel J. Math.},
  FJOURNAL = {Israel Journal of Mathematics},
    VOLUME = {201},
      YEAR = {2014},
    NUMBER = {1},
     PAGES = {373--414},
     
}

@article {MR3683097,
    AUTHOR = {Arhancet, C\'edric and Fackler, Stephan and Le Merdy,
              Christian},
     TITLE = {Isometric dilations and {$H^\infty$} calculus for bounded
              analytic semigroups and {R}itt operators},
   JOURNAL = {Trans. Amer. Math. Soc.},
  FJOURNAL = {Transactions of the American Mathematical Society},
    VOLUME = {369},
      YEAR = {2017},
    NUMBER = {10},
     PAGES = {6899--6933},
     
}

@article {MR458230,
    AUTHOR = {Akcoglu, M. A. and Sucheston, L.},
     TITLE = {Dilations of positive contractions on {$L\sb{p}$} spaces},
   JOURNAL = {Canad. Math. Bull.},
  FJOURNAL = {Canadian Mathematical Bulletin. Bulletin Canadien de
              Math\'ematiques},
    VOLUME = {20},
      YEAR = {1977},
    NUMBER = {3},
     PAGES = {285--292},
     
}

@book {MR481928,
    AUTHOR = {Coifman, Ronald R. and Weiss, Guido},
     TITLE = {Transference methods in analysis},
    SERIES = {Conference Board of the Mathematical Sciences Regional
              Conference Series in Mathematics},
    VOLUME = {No. 31},
 PUBLISHER = {American Mathematical Society, Providence, RI},
      YEAR = {1976},
     PAGES = {ii+59},
      
}

@article {MR615568,
    AUTHOR = {Peller, V. V.},
     TITLE = {Analogue of {J}. von {N}eumann's inequality, isometric
              dilation of contractions and approximation by isometries in
              spaces of measurable functions},
      NOTE = {Spectral theory of functions and operators, II},
   JOURNAL = {Trudy Mat. Inst. Steklov.},
  FJOURNAL = {Akademiya Nauk SSSR. Trudy Matematicheskogo Instituta imeni V.
              A. Steklova},
    VOLUME = {155},
      YEAR = {1981},
     PAGES = {103--150, 185},
      
}

\end{document}